\documentclass[aos]{imsart}

\usepackage{array, multirow}
\usepackage{a4wide}
\usepackage[english]{babel}
\usepackage[T1]{fontenc}
\usepackage[latin1]{inputenc}
\usepackage{dsfont}
\usepackage{mathrsfs}
\usepackage{amsmath}
\usepackage{amssymb}
\usepackage{verbatim}
\usepackage{amsthm}
\usepackage{graphicx}
\usepackage{xcolor}
\usepackage{fancybox}
\usepackage[square,numbers]{natbib}
\usepackage{enumitem}
\usepackage{eurosym}

\usepackage{amsmath, amsfonts, amssymb, amsthm}

\startlocaldefs

\def\a{\alpha}
\def\b{\beta}
\def\g{\gamma}

\def\d{\delta}

\def\e{\epsilon}

\def\t{\theta}
\def\T{\Theta}

\def\l{\lambda}

\def\s{\sigma}

\def\ie{\textit{i.e., }}

\def\cf{\textit{cf. }}

\def\RR{\mathbb R}

\def\calA{\mathcal A}
\def\calB{\mathcal B}

\def\fcar{\mathds{1}}

\def\suchthat{\,|\,}

\def\esp{\mathbf E}

\def\prob{\mathbf P}
\def\calN{\mathcal N}
\def\simiid{\overset{iid}{\sim}}

\theoremstyle{plain}
\newtheorem{theorem}{Theorem}
\newtheorem{lemma}{Lemma}
\newtheorem{proposition}{Proposition}

\newtheorem*{theorem*}{Theorem}
\newtheorem*{lemma*}{Lemma}
\newtheorem*{proposition*}{Proposition}
\newtheorem*{corollary*}{Corollary}
\theoremstyle{remark}

\newtheorem*{remark*}{Remark}

\newtheorem*{note*}{Note}
\theoremstyle{definition}

\newtheorem*{definition*}{Definition}

\def\bt{{\boldsymbol \t}}
\def\bY{{\boldsymbol Y}}
\def\bxi{{\boldsymbol \xi}}
\def\bu{{\boldsymbol u}}

\endlocaldefs

\begin{document}

\begin{frontmatter}

\title{Adaptive robust estimation in sparse vector model}
\runtitle{Adaptive robust estimation in sparse vector model}

\begin{aug}
\author{\fnms{L.} \snm{Comminges$^1$}\ead[label=e2]{laetitia.comminges@dauphine.fr}},
\author{\fnms{O.} \snm{Collier$^{2}$}\corref{}\ead[label=e1]{olivier.collier@parisnanterre.fr}},
\author{\fnms{M.} \snm{Ndaoud$^3$}\ead[label=e3]{ndaoudm@gmail.com}}
\and
\author{\fnms{A.B.} \snm{Tsybakov$^3$}\ead[label=e4]{alexandre.tsybakov@ensae.fr}}

\address{CEREMADE, Universit\'e Paris-Dauphine\\ PSL Research University \\ 75016 Paris, France\\ \printead{e2}}
\affiliation{$^1$CEREMADE, Universit\'e Paris-Dauphine, PSL and CREST}

\address{Modal'X, UPL, Universit\'e Paris Nanterre\\ 92000 Nanterre France\\ \printead{e1}}
\affiliation{$^2$Modal'X, UPL, Universit\'e Paris Nanterre and CREST}

\address{CREST (UMR CNRS 9194), ENSAE \\5, av. Henry Le Chatelier, 91764 Palaiseau, France\\ \printead{e3,e4}}
\affiliation{$^3$CREST, ENSAE}

\end{aug}

\runauthor{Comminges, L., Collier, O., Ndaoud, M. and Tsybakov, A.B.}

\begin{abstract} \
For the sparse vector model, we consider estimation of the target vector, of its $\ell_2$-norm and of 
the noise variance. We construct adaptive estimators and establish the optimal rates of adaptive 
estimation when adaptation is considered  with respect to the triplet "noise level -- 
noise distribution -- sparsity". We consider classes of noise distributions with polynomially and
exponentially decreasing tails as well as the case of Gaussian noise.  
The obtained rates turn out to be different from the minimax non-adaptive rates when the triplet 
is known. A crucial issue is the ignorance of the noise variance. Moreover, knowing or not 
knowing the noise distribution can also influence the rate. For example, the rates of 
estimation of the noise variance can differ depending on whether the noise is Gaussian or 
sub-Gaussian without a precise knowledge of the distribution.  Estimation of noise variance 
in our setting can be viewed as an adaptive variant of robust estimation of scale in the 
contamination model, where instead of fixing the "nominal" distribution in advance we assume 
that it belongs to some class of distributions. 
\end{abstract}



\end{frontmatter}


\section{Introduction}

This paper considers estimation of the unknown sparse vector, of its $\ell_2$-norm and of the noise level in the sparse sequence model. The focus is on construction of estimators that are optimally adaptive in a minimax sense with respect to the noise level, to the form of the noise distribution, and to the sparsity. 

We consider the model defined as follows. Let the signal $\bt=(\t_1,\ldots,\t_d)$ be observed with noise of unknown magnitude $\s>0$:
\begin{equation}\label{model}
	Y_i = \t_i + \s\xi_i, \quad i=1,\ldots,d.
\end{equation}
The noise random variables $\xi_1,\ldots,\xi_d$ are assumed to be i.i.d. and we denote by 
 $P_\xi$ the unknown distribution of $\xi_1$. We assume throughout that the noise is zero-mean, 
$\esp(\xi_1)=0$, and  that $\esp(\xi_1^2)=1$, since $\s$ needs to be identifiable. We denote 
by $\prob_{\bt,P_\xi,\s}$ the distribution of $\bY=(Y_1,\dots,Y_d)$ when the signal is $\bt$, the noise level is $\s$ and the distribution of the noise variables is $P_\xi$.  We also denote by $\esp_{\bt,P_\xi,\s}$ the expectation with respect to $\prob_{\bt,P_\xi,\s}$. 

We assume that the signal $\bt$ is $s$-sparse, \ie
\begin{equation*}
	\|\boldsymbol{\t}\|_0 = \sum_{i=1}^d \fcar_{\t_i\neq0} \le s,
\end{equation*}
where $s\in \{1,\dots, d\}$ is an integer. Set $\T_s=\{\bt\in\RR^d\suchthat\|\bt\|_0\le s\}$. We consider the problems of estimating 
$\bt$ under the $\ell_2$ loss,  estimating the variance $\sigma^2$, and estimating the $\ell_2$-norm
 $$\|\boldsymbol{\t}\|_2 = \Big(\sum_{i=1}^d \t_i^2\Big)^{1/2}.$$

This setting arises in several applications (spectroscopy, astronomy, interferometry), as discussed, for example in \cite{DJ}, where a typical observation is an almost zero signal with rare spikes, which is sometimes called a nearly black object.  Theoretical work on this model mainly focuses on the case of Gaussian and sub-Gaussian noise with known $\sigma$ while in practice other types of noise may be relevant and the noise variance is usually unknown. The classical Gaussian sequence model corresponds to the case where the noise $\xi_i$ is standard Gaussian ($P_\xi={\mathcal N}(0,1)$) and the noise level $\s$ is known. Then, the optimal rate of estimation of $\bt$ under the $\ell_2$ loss in a minimax sense on the class $\T_s$ is $\sqrt{s\log(ed/s)}$  and it is attained by thresholding estimators \cite{DJ}. Also, for the Gaussian sequence model with known $\s$, minimax optimal estimator of the norm $\|\bt\|_2$  as well as the corresponding minimax rate are available from~\cite{CollierCommingesTsybakov2017}  (see Table 1). 

In this paper, we study estimation of the three objects $\bt$,  $\|\boldsymbol{\t}\|_2$, and $\s^2$ 
in the following two settings.
\begin{itemize}

\item[(a)] {\it The distribution of $\xi_i$ and the noise level $\s$ are both unknown.} This is our main setting of interest.
For the unknown distribution of $\xi_i$, we consider two types of assumptions. Either $P_\xi$ belongs to a class $\mathcal{G}_{a,\tau}$, \ie for some $a,\tau>0$, 
\begin{equation}\label{definition_subgaussian}
	P_\xi\in\mathcal{G}_{a,\tau} \quad \text{ iff} \quad \esp(\xi_1)=0, \ \esp(\xi_1^2)=1 \ \text{and} \quad \forall t\ge2, \ \prob\big(|\xi_1|>t\big) \le 2 e^{-(t/\tau)^a},
\end{equation}
which includes for example sub-Gaussian distributions ($a=2$), or to a class of distributions with polynomially decaying tails $\mathcal{P}_{a,\tau}$, \ie for some $\tau>0$ and $a\ge  2$, 
\begin{equation}\label{definition_polynomial}
P_\xi\in\mathcal{P}_{a,\tau} \quad \text{ iff } \quad \esp(\xi_1)=0, \ \esp(\xi_1^2)=1  \ \text{and} \quad \forall t\ge2, \ \prob\big(|\xi_1|> t) \le \Big(\frac{\tau}{t}\Big)^a.
\end{equation}
We propose estimators of $\bt$,  $\|\boldsymbol{\t}\|_2$, and $\s^2$ that are optimal in non-asymptotic minimax sense on these classes of distributions and the sparsity class $\Theta_s$. We establish the corresponding non-asymptotic minimax rates. 
They are given in the second and third columns of Table~1. We also provide the minimax optimal estimators.

\item[(b)] {\it Gaussian noise $\xi_i$ and unknown $\s$.} The results on the non-asymptotic minimax 
rates are summarized in the first column of Table 1.
Notice an interesting effect -- the rates of estimation of $\s^2$ and of the norm $\|\bt\|_2$ when the noise is Gaussian are faster than the optimal rates when the noise is sub-Gaussian. This can be seen by comparing the first column of Table 1 with the particular case $a=2$ of the second column corresponding to sub-Gaussian noise.
\end{itemize}

\begin{table}[h!]
       \label{tab:table1}
    \begin{tabular}{|c|c|c|c|c|}\hline
      &{Gaussian noise} & {Noise in class $\mathcal{G}_{a,\tau}$} & {Noise in class $\mathcal{P}_{a,\tau}$}\\
      &               model                  &                       &             \\
      \hline
       &         &    {\multirow{5}{*}{$\sqrt{s}\log^{\frac1a}(ed/s)$}}& {\multirow{5}{*}{$\sqrt{s}(d/s)^{\frac1a}$}}\\ 
     $\bt$ &  $\sqrt{s\log(ed/s)}$ &   {\multirow{9}{*}{}}& {\multirow{5}{*}{}} \\ 
              &known $\s$ \cite{DJ}   &    {\multirow{5}{*}{}}& {\multirow{5}{*}{}}\\ 
              &unknown $\s$ \cite{Verzelen2012}  &    unknown $\s$& unknown $\s$\\ 
                &  &    {\multirow{5}{*}{}}& {\multirow{5}{*}{}}\\ 
                 \hline
                  &    &    {\multirow{5}{*}{}}& {\multirow{5}{*}{}}\\
 $\|\bt\|_2$ &   $\sqrt{s\log(1+\frac{\sqrt{d}}{s})} \wedge d^{1/4}$         &   $\sqrt{s}\log^{\frac1a}(ed/s)\wedge d^{1/4}$  &$\sqrt{s}(d/s)^{\frac1a} \wedge d^{1/4}$\\
      &   known $\s$ \cite{CollierCommingesTsybakov2017}   &    known $\s$ & known $\s$ \\ 
      &   $ \sqrt{s\log(1+\frac{\sqrt{d}}{s})} \vee \sqrt{\frac{s}{1+\log_+(s^2/d)}}$  &    $\sqrt{s}\log^{\frac1a}(ed/s)$ & $\sqrt{s}(d/s)^{\frac1a}$ \\ 
      &  unknown $\s$    &    unknown $\s$ & unknown $\s$ \\ 
       \hline
             &  {\multirow{3}{*}{$\displaystyle{ 
             \frac1{\sqrt{d}} \vee \frac{s}{d(1+\log_+(s^2/d))}}$ 
              }
                    }
                            &  {\multirow{3}{*}{$\displaystyle{\frac1{\sqrt{d}}\vee \frac{s}{d}\log^{\frac{2}{a
                            }}\left(\frac{ed}{s}\right)}$}}  
                            &  {\multirow{3}{*}{$\displaystyle{\frac1{\sqrt{d}} \vee \Big(\frac{s}{d}\Big)^{1-\frac2a}}$}}\\
$\s^2$      &    &    {\multirow{3}{*}{}}& {\multirow{3}{*}{}}\\ 
                 &    &    {\multirow{3}{*}{}}& {\multirow{3}{*}{}}\\ 
      \hline
    \end{tabular}
    \vspace{3mm}
     \caption{\rm Optimal rates of convergence.}
\end{table}

Some comments about Table 1 and additional details are in order. 
\begin{itemize}
\item The difference between the minimax rates for estimation of $\bt$ and estimation of the $\ell_2$-norm $\|\bt\|_2$ turns out to be specific for the pure Gaussian noise model. It disappears for the classes $\mathcal{G}_{a,\tau}$ and $\mathcal{P}_{a,\tau}$.
In particular, if the noise is in one of these classes and $\s$ is unknown, the minimax rate 
for $\|\bt\|_2$ does not have an elbow between the "dense" ($s>\sqrt{d}$) and the "sparse" ($s\le \sqrt{d}$) zones, in contrast to the Gaussian case.
\item For the problem of estimation of variance $\s^2$ with {\it known} distribution of the noise $P_\xi$, we consider a more general setting than (b) mentioned above.  We show that when the noise distribution is  exactly known (and satisfies a rather general assumption, not necessarily Gaussian - can have polynomial tails), then the rate of estimation of $\sigma^2$ can be as fast as $\max\left(\frac1{\sqrt{d}}, \frac{s}{d}\right)$, which is faster than the optimal rate $\max\left(\frac1{\sqrt{d}}, \frac{s}{d}\log\left(\frac{ed}{s}\right)\right)$ for the class of sub-Gaussian noise. In other words, the phenomenon of improved rate is not due to the Gaussian character of the noise but rather to the fact that the noise distribution is known.
\item Our findings show that there is a dramatic difference between the behavior of optimal estimators of $\bt$ in the sparse sequence model and in the sparse linear regression model with "well spread" regressors.  It is known from  \cite{Gautier2013, Belloni2014}  that in sparse linear regression with "well spread" regressors (that is, having positive variance), the rates of estimating $\bt$ are the same for the noise with sub-Gaussian and polynomial tails. In particular, Theorem 1 in \cite{Gautier2013} shows that the self-tuning Dantzig estimator attains sub-Gaussian rates uniformly over the class of symmetric error distributions with bounded absolute moment of order $2+\delta$ for $\delta>0$. 
We show that the situation is quite different in the sparse sequence model, where the optimal rates are much slower and depend on the polynomial index of the noise. 
\end{itemize}

%

We conclude this section by a discussion of related work.  \citet*{ChenGaoRen2015} explore the problem of robust estimation of variance and of covariance matrix under Hubers's contamination model. As explained in Section \ref{sec:variance} below, this problem has similarities with estimation of noise level in our setting. The main difference is
that instead of fixing in advance the Gaussian nominal distribution of the contamination model we assume that it belongs to a class of distributions, such as \eqref{definition_subgaussian} or \eqref{definition_polynomial}. Therefore, 
the corresponding results in Section \ref{sec:variance} can be viewed as results  on 
robust estimation of scale where, in contrast to the classical setting, we are interested in adaptation to the unknown nominal law. Another aspect of robust estimation of scale is analyzed by \citet{Minsker and Wei2017} who consider classes of distributions similar to $\mathcal{P}_{a,\tau}$ rather than the contamination model. The main aim in \cite{Minsker and Wei2017} is to construct estimators having sub-Gaussian deviations under weak moment assumptions.
Our setting is different in that we consider the sparsity class $\T_s$ of vectors 
$\bt$ and the rates that we obtain depend on $s$. Estimation of variance in sparse linear model 
is discussed in \cite{SunZhang2012} where some
upper bounds for the rates are given.  
We also mention the recent paper \cite{GolubevKrymova2017} that deals with estimation 
of variance in linear regression in a framework that does not involve sparsity, as well as 
the work on estimation of signal-to-noise ratio functionals in settings involving sparsity \cite{VerzelenGassiat2018,GuoCai2018} 
and not involving sparsity \cite{JansonBarberCandes2017}. Papers \cite{CollierCommingesTsybakovVerzelen2017,CarpentierVerzelen2019} discuss estimation 
of other functionals than the $\ell_2$-norm $\|\bt\|_2$ in the sparse vector model when the noise is Gaussian with 
unknown variance, while \cite{Carpentier at al2019} considers estimation of $\|\bt\|_2$ in sparse linear regression when the noise and covariates are Gaussian and the noise variance is known.

{\bf Notation.} For $x>0$, let $\lfloor  x \rfloor$ denote the maximal integer smaller than $x$. 
For a finite set $A$, we denote by $|A|$ its cardinality. Let $\inf_{\hat{T}}$ denote the infimum over all estimators.
The notation $C$, $C^{\prime}$,$c$, $c^{\prime}$ will be used 
for positive constants that can depend only $a$ and 
$\tau$ and can vary from line to line.

\section{Estimation of sparse vector $\bt$}\label{sec:vector}

In this section, we study the problem of  estimating  a sparse vector $\bt$ in $\ell_2$-norm when the variance of noise $\s$ and the  distribution of $\xi_i$ are both unknown. We only assume that the noise distribution belongs to a given class, which can be either a class of distributions with polynomial tails $\mathcal{P}_{a,\tau} $, or a class $\mathcal{G}_{a,\tau} $ with exponential decay of the tails.

First, we introduce a preliminary estimator  $\tilde{\s}^2$ of $\s^2$ that will be used to define an estimator of~$\bt$.  
Let $\g\in(0,1/2]$ be a constant that will be chosen small enough and depending only on $a$ and~$\tau$. 
Divide $\{ 1,\ldots, d\}$ into $m=\lfloor  \g d \rfloor$ disjoint subsets $B_1,\dots,B_m$, each 
of cardinality $|B_i|\ge k:=\lfloor d/m \rfloor \ge {1}/{\g}-1$. 
 Consider the median-of-means estimator 
 \begin{equation}\label{mom}
  \tilde{\s}^2 = {\sf med}(\bar{\s}_1^2,\dots, \bar{\s}_m^2), 
\ \text{where} \ \bar{\s}_i^2= \frac{1}{|B_i|} \sum_{j\in B_i} Y_j^2, \quad i=1,\dots,m.
 \end{equation}
 Here, ${\sf med}(\bar{\s}_1^2,\dots, \bar{\s}_m^2)$ denotes the median of $\bar{\s}_1^2,\dots, \bar{\s}_m^2$. 
The next proposition shows that the estimator $\tilde{\s}^2$ recovers $\s^2$
to  within a constant factor. 

\begin{proposition}\label{proposition_over}
{
 	Let $\tau>0,a>2$. There exist  constants $\g\in(0,1/2]$, $c>0$ and $C>0$  depending only on $a$ and $\tau$  such that for any integers  $s$ and $d$ satisfying $1\le s<  \lfloor \g d\rfloor /4$  we have
		\begin{equation*}
		\inf_{P_\xi\in\mathcal{P}_{a,\tau}} \inf_{\s>0} \inf_{\|\bt\|_0\le s}\prob_{\bt,P_\xi,\s} \Big( 1/2\le  \frac{\tilde{\s}^2}{\s^2}\le 3/2\Big) \ge  1-\exp(-c d),
	\end{equation*}
	\begin{equation*}
		\sup_{P_\xi\in\mathcal{P}_{a,\tau}} \sup_{\s>0} \sup_{\|\bt\|_0\le s} \frac{\esp_{\bt,P_\xi,\s} \left| \tilde{\s}^2 - \s^2 \right|}{\s^{2}} \le  C,
	\end{equation*}
	and for $a>4$,
	\begin{equation*}
		\sup_{P_\xi\in\mathcal{P}_{a,\tau}} \sup_{\s>0} \sup_{\|\bt\|_0\le s}\frac{\esp_{\bt,P_\xi,\s} \left( \tilde{\s}^2 - \s^2 \right)^{2}}{\s^{4}} \le  C.
	\end{equation*}
	}
\end{proposition}
Note that the result of Proposition~\ref{proposition_over} also holds for the class $\mathcal{G}_{a,\tau}$ for all $a>0$ and $\tau>0$. Indeed, $\mathcal{G}_{a,\tau}\subset\mathcal{P}_{a,\tau}$ for all $a> 2$ and $\tau>0$, while for any $0<a\le 2$ and $\tau>0$, there exist $a'> 4$ and $\tau'>0$ such that $\mathcal{G}_{a,\tau}\subset\mathcal{P}_{a',\tau'}$.

We further note that assuming $s< cd$ for some $0<c<1$ is natural in the context of variance estimation
since $\s$ is not identifiable when $s=d$. In what follows, all upper 
bounds on the risks of estimators will be obtained under this assumption. 

Consider now an estimator ${\hat \bt}$  defined as follows:
\begin{equation}\label{def_estimateur_mom}
	\hat{\bt} \in \text{arg}\min_{\bt\in \RR^d} \Big(\sum_{i=1}^d (Y_i-\t_i)^2+\tilde{\s}\|\bt\|_*\Big).
\end{equation}
Here,  $\|\cdot\|_*$ denotes the sorted $\ell_1$-norm:
 \begin{equation}\label{sorted_norm}
\|\bt\|_*=\sum_{i=1}^d \l_i |\t|_{(d-i+1)},
\end{equation}
 where $|\t|_{(1)}\le \cdots\le |\t|_{(d)}$ are the order statistics of  $|\t_1|,\ldots,|\t_d|$,  and $\l_1\ge\cdots\ge\l_p>0$ are tuning parameters.
 
Set
 \begin{equation}\label{opt_rate}
\phi_{\sf exp}^*(s,d)=\sqrt{s}\log^{1/a}(ed/s), \qquad \phi_{\sf pol}^*(s,d)= \sqrt{s}(d/s)^{1/a}.
\end{equation}
The next theorem shows that $\hat{\bt}$ estimates $\bt$  with the rates $\phi_{\sf exp}^*(s,d)$ and $\phi_{\sf pol}^*(s,d)$ when the noise distribution belongs to the class $\mathcal{G}_{a,\tau} $ and class $\mathcal{P}_{a,\tau} $, respectively.
\begin{theorem}\label{theorem_adaptiveupperbound}
	Let $s$ and $d$ be integers satisfying $1\le s< \lfloor \g d\rfloor/4$ where 
$\g\in(0,1/2]$ is the tuning parameter in the definition of $\tilde \s^2$. 
Then for the estimator $\hat{\bt}$ defined by \eqref{def_estimateur_mom} the following holds.
\begin{enumerate}
{
\item	 Let $\tau>0$, $a>0$.  There exist constants
	 $c,C>0$ and $\g\in(0,1/2]$ depending only on $(a,\tau)$  such that if 
$\l_j=c  \log^{1/a}(ed/j),  j=1,\ldots, d,$ we have
	\begin{equation*}
		\sup_{ P_\xi\in\mathcal{G}_{a,\tau}}\sup_{\s>0}\sup_{\|\bt\|_0\le s} 
		\frac{\esp_{\bt,P_\xi,\s}
		\left(\|\hat{\bt}-\bt\|^{2}_2 \right)}{\s^{2}} \le C\left(\phi_{\sf exp}^{*}(s,d)\right)^2.	
	\end{equation*}
\item	 Let $\tau>0,a>2$. There exist constants
	 $c,C>0$ and $\g\in(0,1/2]$ depending only on $(a,\tau)$ such that if 	
$\l_j=c  ({d}/{j})^{1/a},  j=1,\ldots, d,$ we have
	\begin{equation*}
		\sup_{ P_\xi\in\mathcal{P}_{a,\tau}}\sup_{\s>0}\sup_{\|\bt\|_0\le s} 
		\frac{\esp_{\bt,P_\xi,\s}
		\left(\|\hat{\bt}-\bt\|^{2}_2 \right)}{\s^{2}} \le C\left(\phi_{\sf pol}^{*}(s,d)\right)^2.	
	\end{equation*}
}
\end{enumerate}
\end{theorem}

Furthermore, it follows from the lower bound of Theorem~\ref{theorem_lowerbound_norm_subgaussian}
in Section \ref{sec:l2_norm} that the rates $\phi_{\sf exp}^*(s,d)$ and $\phi_{\sf pol}^*(s,d)$ cannot be improved in a minimax sense. Thus, the estimator $\hat{\bt}$ defined in \eqref{def_estimateur_mom} achieves the optimal rates in a minimax sense. 

Note that the rates in Theorem \ref{theorem_adaptiveupperbound} and also all the rates shown in Table 1 for the classes $\mathcal{G}_{a,\tau}$ and $\mathcal{P}_{a,\tau}$ are achieved on estimators that are adaptive to the sparsity index $s$. Thus, knowing or not knowing $s$ does not influence the optimal rates of estimation when the distribution of $\xi$ and the noise level are unknown.

From Theorem \ref{theorem_adaptiveupperbound}, we can conclude that the optimal rate $\phi_{\sf pol}^*$ under  polynomially decaying noise is very different from the optimal rate $\phi_{\sf exp}^*$ under exponential tails, in particular, from the rate under the sub-Gaussian noise. At first sight, this phenomenon seems to contradict some results in the literature on sparse regression model. 
Indeed, \citet*{Gautier2013}  consider sparse linear regression with unknown noise level $\s$ and show that the Self-Tuned Dantzig estimator can achieve the same rate as in the case of Gaussian noise (up to a logarithmic factor) under the assumption that the noise is symmetric and has only a bounded moment of order $a>2$. \citet*{Belloni2014} show for the same model that a square-root Lasso estimator achieves analogous behavior under the assumption that the noise has a bounded moment of order $a>2$. However, a crucial condition in 
\cite{Belloni2014} is that the design is "well spread", that is all components of the design vectors are random with positive variance. The same type of condition is needed in \cite{Gautier2013} to obtain a sub-Gaussian rate. This condition of "well spreadness" is not satisfied in the sparse sequence model that we are considering here. In this model viewed as a special case of linear regression, the design is deterministic, with only one non-zero component. We see that such a degenerate design turns out to be the least favorable from the point of view of the convergence rate, while the "well spread" design is the best one. An interesting general conclusion of comparing our findings to \cite{Gautier2013} and \cite{Belloni2014} is that the optimal rate of convergence of estimators under sparsity when the noise level is unknown depends dramatically on the properties of the design. There is a whole spectrum of possibilities between the degenerate and "well spread" designs where a variety of new rates can arise depending on the properties of the design. Studying them remains an open problem.

\section{Estimation of the $\ell_2$-norm}\label{sec:l2_norm}

In this section, we consider the problem of estimation of the $\ell_2$-norm of a sparse vector when the variance of the noise and the form of its distribution are both unknown. We show
that 
the rates $\phi_{\sf exp}^*(s,d)$ and $\phi_{\sf pol}^*(s,d)$ are optimal in a minimax sense on the classes $\mathcal{G}_{a,\tau}$ and $\mathcal{P}_{a,\tau}$, respectively.  We first provide a lower bound on the risks of any estimators of the $\ell_2$-norm when the noise level $\s$ is unknown and the unknown noise distribution $P_\xi$ belongs either to $\mathcal{G}_{a,\tau}$ or $\mathcal{P}_{a,\tau}$.
{
We denote by $\mathcal L$ the set of all monotone non-decreasing functions $\ell:[0, \infty)\to [0, \infty)$ such that
$\ell(0)=0$ and $\ell\not\equiv 0$. 
}
\begin{theorem}\label{theorem_lowerbound_norm_subgaussian}
Let $s,d$ be integers satisfying $1\le s\le d$. Let $\ell(\cdot)$ be any loss function in the class $\mathcal L$.  Then, for any $a>0,\tau>0$, 	
	\begin{equation}\label{lowerbound1:norm}
	\inf_{\hat{T}}  \sup_{P_\xi\in\mathcal{G}_{a,\tau}} \sup_{\s>0} \sup_{\|\bt\|_0\le s} 
	\esp_{\bt,P_\xi,\s}\,\ell \Big( c(\phi_{\sf exp}^*(s,d))^{-1} \Big| 
\frac{\hat{T}-\|\bt\|_2}{\s}\Big| \Big)\ge c',
	\end{equation}
	and, for any $a\ge 2,\tau>0$, 
	\begin{equation}\label{lowerbound2:norm}
	\inf_{\hat{T}}  \sup_{P_\xi\in\mathcal{P}_{a,\tau}} \sup_{\s>0} \sup_{\|\bt\|_0\le s} 
	\esp_{\bt,P_\xi,\s}\,\ell \Big( \bar c(\phi_{\sf pol}^*(s,d))^{-1} \Big| 
\frac{\hat{T}-\|\bt\|_2}{\s}\Big| \Big)\ge \bar c'.
	\end{equation}
Here,  $\inf_{\hat{T}}$ denotes the infimum over all estimators, and  $c, \bar c>0$, $c', \bar 
c'>0$ are constants that can depend only on $\ell(\cdot)$, $\tau$ and $a$.	
\end{theorem}

The lower bound~\eqref{lowerbound2:norm} implies that the rate of estimation of the $\ell_2$-norm of a sparse vector
deteriorates dramatically if the bounded moment assumption is imposed on the noise instead, for example, of the sub-Gaussian assumption.

Note also that \eqref{lowerbound1:norm}  and \eqref{lowerbound2:norm} immediately imply lower bounds with the same rates $\phi_{\sf exp}^* $ and $\phi_{\sf pol}^*$ for the estimation of the $s$-sparse vector $\bt$ under the $\ell_2$-norm.
	
Given the upper bounds of Theorem  \ref{theorem_adaptiveupperbound}, the lower bounds~\eqref{lowerbound1:norm} and~\eqref{lowerbound2:norm} are tight for the quadratic loss, 
 and are achieved by the following  plug-in estimator independent of $s$ or $\s$:
\begin{equation}\label{definition_norm_mom}
	\hat{N} = \|\hat{\bt}\|_2
\end{equation}
where $\hat{\bt}$ is defined in~\eqref{def_estimateur_mom}. 

In conclusion, when both $P_\xi$ and $\s$ are unknown the rates $\phi_{\sf exp}^* $ and $\phi_{\sf pol}^*$ defined in \eqref{opt_rate} are minimax optimal both for estimation of $\bt$ and of the norm $ \|\bt\|_2$.  

We now compare these results with the findings in~\cite{CollierCommingesTsybakov2017} regarding the (nonadaptive) estimation of  $\|\bt\|_2$ when $\xi_i$ have the standard Gaussian distribution ($P_\xi = {\cal N}(0,1)$) and $\s$ is known.  It is shown in~\cite{CollierCommingesTsybakov2017} that in this case the optimal rate of estimation of $\|\bt\|_2$ has the form 
$$
\phi_{{\cal N}(0,1)}(s,d)= \min\left\{\sqrt{s\log(1+\sqrt{d}/s)},d^{1/4}\right\}.
$$  
Namely, the following proposition holds.  
\begin{proposition}[Gaussian noise, known $\s$ \cite{CollierCommingesTsybakov2017}]\label{prop:lower:gaussian}
	For any $\s>0$ and any integers $s,d$ satisfying $1\le s\le d$, we have
	\begin{equation*}
	c \s^2 
 		\phi_{{\cal N}(0,1)}^2(s,d)\le \inf_{\hat{T}} \sup_{\|\bt\|_0\le s} \esp_{\bt,{\cal N}(0,1),\s} \big(\hat{T}-\|\bt\|_2\big)^2 \le C \s^2 
		\phi_{{\cal N}(0,1)}^2(s,d),
	\end{equation*}
	where $c>0$ and $C>0$ are absolute constants and $\inf_{\hat{T}}$ denotes the infimum over all estimators. 
\end{proposition}

We have seen that, in contrast to this result, in the case of unknown $P_\xi$ and $\s$ the optimal rates \eqref{opt_rate}
do not exhibit an elbow at $s=\sqrt{d}$ between the "sparse" and "dense" regimes. Another conclusion is that, in the "dense" zone $s>\sqrt{d}$, adaptation to $P_\xi$ and $\s$ is only possible with a significant deterioration of the rate. On the other hand, for the sub-Gaussian class $\mathcal{G}_{2,\tau}$, in the "sparse" zone $s\le \sqrt{d}$ the non-adaptive rate $\sqrt{s\log(1+\sqrt{d}/s)}$ differs only slightly from the adaptive sub-Gaussian rate $\sqrt{s\log(ed/s)}$; in fact, this 
difference in the rate appears only in a vicinity of $s=\sqrt{d}$. 

A natural question is whether such a deterioration of the rate is caused by the ignorance of $\s$ or by the ignorance of the distribution of $\xi_i$ within the sub-Gaussian class $\mathcal{G}_{2,\tau}$. 
The answer is that both are responsible. It turns out that if only one of the two ingredients ($\s$ or the noise distribution) is unknown, then 
a rate faster than the adaptive sub-Gaussian rate $\phi_{\sf exp}^*(s,d) = \sqrt{s\log(ed/s)}$ can be achieved. This is detailed in the next two propositions.

{
Consider first the case of Gaussian noise and unknown $\s$. Set
$$
\phi_{{\cal N}(0,1)}^*(s,d)= \max\left\{\sqrt{s\log(1+\sqrt{d}/s)},\sqrt{\frac{s}{1+
\log_+(s^{2}/d)}}\right\},
$$ 
where $\log_+(x)=\max(0,\log(x))$ for any $x>0$.
We divide the set $\{1, \dots, d\}$ into two disjoint subsets $I_{1}$ and $I_{2}$ 
with $\min\left(|I_{1}|,|I_{2}|\right)\geq \lfloor {d}/{2}\rfloor$. Let $\hat{\s}^{2}$ be the variance estimator defined by 
\eqref{definition_noisevarianceestimator_gauss}, cf. Section \ref{sec:median} below, and let 
$\hat{\s}^{2}_{\sf med,1}, \hat{\s}^{2}_{\sf med,2}$ be the median estimators 
\eqref{definition_median} corresponding to the samples $(Y_{i})_{i \in I_{1}}$ and 
$(Y_{i})_{i \in I_{2}}$, respectively. 
Consider the estimator
\begin{equation}\label{eq:C}
\hat{N}^* = \left\{
  \begin{array}{lcl}
\sqrt{ \Big| \sum_{j=1}^d (Y_j^2~\fcar_{\{ |Y_j|>\rho_{j} \}}) -d\a \hat{\s}^2\Big|}& \text{if}& s\le \sqrt{d},\\
 \sqrt{ \Big|  \sum_{j=1}^d Y_j^2 -d \hat{\s}^2\Big|}\phantom{~\fcar_{\{ |Y_j|>\
 \rho \}}}& \text{if} & s> \sqrt{d},
  \end{array}
  \right.
\end{equation} 
where $\rho_{j}= 2 \hat{\s}_{\sf med,1}  \sqrt{2\log (1+d/s^2)}$ if $j \in I_{2}$, $\rho_{j}= 
2 \hat{\s}_{\sf med,2}  \sqrt{2\log (1+d/s^2)}$ if $j \in I_{1}$ and $\a = \esp\left(\xi_1^2~\fcar_{\{ |\xi_1|>2  \sqrt{2\log (1+d/s^2)} \}}\right)$. Note that $Y_j$ is independent of $\rho_j$ for every $j$. Note also that the estimator $\hat{N}^*$
 depends 
on the preliminary estimator ${\tilde \s}^2$ since $\hat{\s}>0$  
defined in \eqref{definition_noisevarianceestimator_gauss} depends on it.

}

\begin{proposition}[Gaussian noise, unknown $\s$]\label{prop:norm:gauss}
The following two properties hold.
\begin{itemize}
\item[(i)] 
{
Let $s$ and $d$ be integers satisfying $1\le s< \lfloor \g d\rfloor/4$, where 
$\g\in(0,1/2]$ is the tuning parameter in the definition of $\tilde \s^2$. 
There exist absolute constants $C>0$ and $\g\in(0,1/2]$ such that	
\begin{equation*}\label{upperbound:norm:gauss}
\sup_{\s>0}\sup_{\|\bt\|_0\le s} \frac{\esp_{\bt, {\cal N}(0,1),\s}  \left( \hat{N}^*-\|\bt\|_2 \right)^{2}}{\sigma^{2}} \le C\left(\phi_{{\cal N}(0,1)}^{*}(s,d)\right)^2.
	\end{equation*}

\item[(ii)] 
Let $s$ and $d$ be integers satisfying $1\le s\le d$ and
let $\ell(\cdot)$ be any loss function in the class $\mathcal L$.  Then, 
	\begin{equation*}\label{lowerbound:norm:gauss}
	\inf_{\hat{T}}  \sup_{\s>0} \sup_{\|\bt\|_0\le s} 
	\esp_{\bt,{\cal N}(0,1),\s}\,\ell \bigg( c(\phi_{{\cal N}(0,1)}^*(s,d))^{-1} \bigg|\frac{ \hat{T}-\|\bt\|_2}{\s}\bigg| \bigg)\ge c^{\prime},
	\end{equation*}
	where  $\inf_{\hat{T}}$ denotes the infimum over all estimators, and $c>0$, $c^{\prime}>0$ are constants that can depend only on $\ell(\cdot)$.
	}
	\end{itemize}
\end{proposition}

The proof of item (ii) of Proposition~\ref{prop:norm:gauss} (the lower bound) is given 
in the Supplementary 
material.

Proposition~\ref{prop:norm:gauss} establishes the minimax optimality of the rate $\phi_{{\cal N}(0,1)}^{*}(s,d)$.
It also shows that if $\s$ is unknown, the knowledge of the Gaussian character of the noise leads to an 
improvement of the rate compared to the adaptive sub-Gaussian rate $\sqrt{s\log(ed/s)}$. However, the 
improvement is only in a logarithmic factor.

{
Consider now the case of unknown noise distribution in $\mathcal{G}_{a,\tau}$ and known $\s$. We show in the next 
proposition  that in this case the minimax rate is of the form  
$$
\phi_{\sf exp}^\circ(s,d)= \min\{\sqrt{s} \log^{\frac{1}{a}}(ed/s),d^{1/4}\}
$$ 
and it is achieved by the estimator
$$
\hat{N}^\circ_{\sf exp} = \left\{
  \begin{array}{lcl}
\phantom{\sum_{j=1}^d Y }\|\hat{\bt}\|_2& \text{if}& s\le \frac{\sqrt{d}}{\log^{\frac{2}{a}}(ed)} ,\\
 \Big|  \sum_{j=1}^d Y_j^2 -d \s^2\Big|^{1/2}& \text{if} & s> \frac{\sqrt{d}}{\log^{\frac{2}{a}}(ed)} ,
  \end{array}
  \right.
$$ 
where $\hat{\bt}$ is defined in~\eqref{def_estimateur_mom}. Note that $\phi_{\sf exp}^\circ(s,d)$ 
can be written equivalently (up to absolute constants) as  $\min\{\sqrt{s}\log^{\frac{1}{a}}(ed),d^{1/4}\}$.
\begin{proposition}[Unknown noise in $\mathcal{G}_{a,\tau}$, known $\s$]\label{prop:norm:known_sigma}  
{
Let $a,\tau>0$.  
The following two properties hold.
}
\begin{itemize}
{
\item[(i)] Let $s$ and $d$ be integers satisfying $1\le s< \lfloor \g d\rfloor/4$, where 
$\g\in(0,1/2]$ is the tuning parameter in the definition of $\tilde \s^2$. 
 There exist constants $c, C>0$, and $\g\in(0,1/2]$ depending only on $(a,\tau)$ such that 
if $\hat{\bt}$ is the estimator defined in~\eqref{def_estimateur_mom} with $\l_j= c\log^{\frac{1}{a}}(ed/j)$
, $j=1,\dots,d$,  then	
\begin{equation*}\label{upperbound:norm:subgauss}
\sup_{P_\xi \in \mathcal{G}_{a,\tau}}\sup_{\|\bt\|_0\le s} \esp_{\bt, P_\xi,\s}  \left(  \hat{N}_{\sf exp}^\circ-\|\bt\|_2 \right)^{2}\le C\sigma^{2} \left(\phi_{\sf exp}^{\circ}(s,d)\right)^2.
	\end{equation*} 
}
	\item[(ii)] Let $s$ and $d$ be integers satisfying $1\le s \le d$
and let $\ell(\cdot)$ be any loss function in the class $\mathcal L$.  Then, there exist constants $c>0$, $c^{\prime}>0$ 
depending only on $\ell(\cdot)$, $a$ and $\tau$ such that 
	\begin{equation*}\label{lowerbound:norm:subgauss}
	\inf_{\hat{T}}  \sup_{P_\xi \in \mathcal{G}_{a,\tau}} \sup_{\|\bt\|_0\le s} 
	\esp_{\bt,P_\xi,\s}\,\ell \bigg( c(\phi_{\sf exp}^\circ(s,d))^{-1} \bigg|
\frac{ \hat{T}-\|\bt\|_2}{\s}\bigg| \bigg)\ge c',
	\end{equation*}
	where  $\inf_{\hat{T}}$ denotes the infimum over all estimators.
	\end{itemize}
\end{proposition}
Proposition~\ref{prop:norm:known_sigma} establishes the minimax optimality of the rate $\phi_{\sf exp}^\circ(s,d)$.
It also shows that if the noise distribution is unknown and belongs to $\mathcal{G}_{a,\tau}$, 
the knowledge of $\s$ leads to an 
improvement of the rate compared to the case when $\s$ is unknown.
In contrast to the case of Proposition~\ref{prop:norm:gauss} (Gaussian noise), the 
improvement here is substantial; it results not only in a logarithmic but in a polynomial factor 
in the dense zone  $s> \frac{\sqrt{d}}{\log^{\frac{2}{a}}(ed)}$.

We end this section by considering the case of unknown polynomial noise and known $\s$. 
The next proposition shows that in this case the minimax rate, for a given $a>4$, is of the form  
$$
\phi_{\sf pol}^\circ(s,d)= \min\{\sqrt{s} (d/s)^{\frac{1}{a}},d^{1/4}\}
$$ 
and it is achieved by the estimator
$$
\hat{N}^\circ_{\sf pol} = \left\{
  \begin{array}{lcl}
\phantom{\sum_{j=1}^d Y }\|\hat{\bt}\|_2& \text{if}& s\le d^{\frac{1}{2}-\frac{1}{a-2}} ,\\
 \Big|  \sum_{j=1}^d Y_j^2 -d \s^2\Big|^{1/2}& \text{if} & s> d^{\frac{1}{2}-\frac{1}{a-2}} ,
  \end{array}
  \right.
$$ 
where $\hat{\bt}$ is defined in~\eqref{def_estimateur_mom}.
\begin{proposition}[Unknown noise in $\mathcal{P}_{a,\tau}$, known $\s$]\label{prop:norm:poly:known_sigma}  
{
Let $\tau>0, a>4$.  
The following two properties hold.
}
\begin{itemize}
{
\item[(i)] Let $s$ and $d$ be integers satisfying $1\le s< \lfloor \g d\rfloor/4$, where 
$\g\in(0,1/2]$ is the tuning parameter in the definition of $\tilde \s^2$. 
 There exist constants $c, C>0$, and $\g\in(0,1/2]$ depending only on $(a,\tau)$ such that 
if $\hat{\bt}$ is the estimator defined in~\eqref{def_estimateur_mom} with 
$\l_j= c(d/j)^{\frac{1}{a}}$, $j=1,\dots,d$,  then
\begin{equation*}\label{upperbound:norm:poly}
\sup_{P_\xi \in \mathcal{P}_{a,\tau}}\sup_{\|\bt\|_0\le s} \esp_{\bt, P_\xi,\s}  \left(  \hat{N}_{\sf pol}^\circ-\|\bt\|_2 \right)^{2}\le C\sigma^{2} \left(\phi_{\sf pol}^{\circ}(s,d)\right)^2.
	\end{equation*} 
}
	\item[(ii)] 
Let $s$ and $d$ be integers satisfying $1\le s \le d$ and let $\ell(\cdot)$ be 
any loss function in the class $\mathcal L$.  Then, there exist constants $c>0$, $c^{\prime}>0$ 
depending only on $\ell(\cdot)$, $a$ and $\tau$ such that
 \begin{equation*}\label{lowerbound:norm:poly}
	\inf_{\hat{T}}  \sup_{P_\xi \in \mathcal{P}_{a,\tau}} \sup_{\|\bt\|_0\le s} 
	\esp_{\bt,P_\xi,\s}\,\ell \bigg( c(\phi_{\sf pol}^\circ(s,d))^{-1} 
\bigg|\frac{ \hat{T}-\|\bt\|_2}{\s}\bigg| \bigg)\ge c^{\prime},
	\end{equation*}
	where  $\inf_{\hat{T}}$ denotes the infimum over all estimators. 	\end{itemize}
\end{proposition}
}
Note that here, similarly to  
Proposition~\ref{prop:norm:known_sigma}, the 
improvement over the case of unknown $\s$ is  in a polynomial factor 
in the dense zone  $s> d^{\frac{1}{2}-\frac{1}{a-2}}$.


\section{Estimating the variance of the noise}\label{sec:variance}

\subsection{Estimating $\s^2$ when the distribution $P_\xi$ is known}\label{sec:median}

In the sparse setting when $\|\bt\|_0$ is small, estimation of the noise level  can be viewed as a problem of robust estimation of scale. Indeed, our aim is to recover  the second moment of~$\s\xi_1$ but the sample second moment cannot be used as an estimator because of the presence of a small number of outliers~$\t_i\ne 0$. Thus, the models in robustness and sparsity problems are quite similar but the questions of interest are different. When robust estimation of $\s^2$ is considered, the object of interest is the pure noise component of the sparsity model while the non-zero components $\t_i$ that are of major interest in the sparsity model play a role of nuisance. 

In the context of robustness, it is known that the estimator based on sample median can be successfully applied. Recall that, when $\bt=0$, the median $M$-estimator of scale (\cf \cite{Huber1981}) is defined as
\begin{equation}\label{definition_median}
	\hat{\s}_{\sf med}^2 = \frac{\hat{M}}{\b}
\end{equation}
where $\hat{M}$ is the sample median of $(Y_1^2,\dots,Y_d^2)$, that is
\begin{equation*} 
	\hat{M} \in  \arg\min_{x>0} \big|F_d(x)-1/2\big|, 
\end{equation*}
and $\beta$ is the median of the distribution of $\xi_1^2$. Here, $F_d$ denotes the empirical c.d.f. of $(Y_1^2,\dots,Y_d^2)$. { Denoting by $F$ the c.d.f. of $\xi^{2}_1$, we have 
\begin{equation}\label{equation_quantile}
	\b=F^{-1}(1/2).
\end{equation}
}
%

The following proposition specifies the rate of convergence of the estimator $\hat{\s}_{\sf med}^2$.

\begin{proposition}\label{prop:gao}
{
Let $\xi_1^{2}$ have a c.d.f. $F$ with positive density, and let $\b$ be given by \eqref{equation_quantile}.	There exist constants $\g\in(0,1/8)$, $c>0$, $c_*>0$ and $C>0$ depending only on $F$ such that for any integers  $s$ and $d$ satisfying $1\le s<  \g d$ and any $t>0$ we have 
	\begin{equation*}
		\sup_{\s>0}\sup_{\|\bt\|_0\le s} \prob_{\bt, F,\s} \left( \Big|\frac{\hat{\s}_{\sf med}^2}{\s^2}-1\Big| \ge c_*\left(\sqrt{\frac{t}{d}}+\frac{s}{d}\right)\right) \le 2(e^{-t} + e^{-c d}).
	\end{equation*}
If $\esp|\xi_{1}|^{2+\epsilon}<\infty$ for some $\epsilon>0$, then,
\begin{equation*}
		\sup_{\s>0}\sup_{\|\bt\|_0\le s} \frac{\esp_{\bt, F,\s} \left| \hat{\s}_{\sf med}^2 - \s^{2} \right|}{\s^{2}} \le C\max\left(\frac1{\sqrt{d}}, \frac{s}{d}\right).
	\end{equation*}
}
\end{proposition}

{
The main message of Proposition~\ref{prop:gao} is that the rate of convergence of 
$\hat{\s}_{\sf med}^2$ in probability and in expectation is as fast as 
\begin{equation}\label{rate:gao}
\max\left(\frac1{\sqrt{d}}, \frac{s}{d}\right)
\end{equation}
and it does not depend on $F$ when $F$ varies in a large class.
The role of Proposition~\ref{prop:gao} is to contrast the subsequent results of this section dealing 
with 
unknown distribution of noise and providing slower rates. It emphasizes the fact that the knowledge of 
the noise distribution is crucial as it leads to an improvement of the rate of 
estimating the variance.  

However, the rate \eqref{rate:gao} achieved by the median estimator is not necessarily optimal.
 As shown in the next proposition, in the case of Gaussian noise the optimal rate is even 
better: 
$$
\phi^{\circ}_{{\cal N}(0,1)}(s,d)= \max\left\{\frac{1}{\sqrt{d}},\frac{s}{d(1+\log_{+}(s^{2}/d))}\right\}.
$$ 
This rate is attained by an estimator that we are going to define now. We use the observation that, in the Gaussian case, 
the modulus of the empirical characteristic function $\varphi_d(t)=\frac{1}{d}\sum_{i=1}^d e^{itY_j}$ is within a constant 
factor from the Gaussian characteristic function $\exp(-\frac{t^2\s^{2}}{2})$ for any $t$.
This suggests the estimator 
$$
\tilde{v}^{2} = -\frac{2\log(|\varphi_{d}(\hat{t}_{1})|)}{\hat{t}_{1}^{2}}
$$
with a suitable choice of $t=\hat{t}_{1}$ that we further set as follows:
$$
\hat{t}_{1} = \frac1{2\tilde{\sigma}}\sqrt{\log\big(16(es/\sqrt{d}+1)\big)} ,
$$
where $\tilde{\sigma}$ is the preliminary estimator \eqref{mom} with some tuning parameter $\g\in(0,1/2]$. The final 
variance estimator is defined as a truncated version of $\tilde{v}^{2}$: 
\begin{equation}\label{definition_noisevarianceestimator_gauss}
    \hat{\s}^{2} = \left\{
  \begin{array}{ll}
 \tilde{v}^{2} & \ \text{if} \ |\varphi_{d}(\hat{t}_{1})|> (es/\sqrt{d}+1)^{-1}/4 ,\\
 \tilde{\sigma}^2 & \ \text{otherwise} .
  \end{array}
  \right.
\end{equation}

\begin{proposition}[Gaussian noise]\label{prop:variance:gauss}  The following two properties hold.
\begin{itemize}
\item[(i)] 
{ Let $s$ and $d$ be integers satisfying $1\le s< \lfloor \g d\rfloor/4$, where 
$\g\in(0,1/2]$ is the tuning parameter in the definition of $\tilde \s^2$. 
There exist absolute constants $C>0$ and $\g\in(0,1/2]$ such that the estimator $\hat{\s}^2$ defined in~(\ref{definition_noisevarianceestimator_gauss}) satisfies
\begin{equation*}\label{upperbound:variance:gauss}
\sup_{\s>0}\sup_{\|\bt\|_0\le s} \frac{\esp_{\bt, {\cal N}(0,1),\s}  
\left| \hat{\s}^2-\s^{2} \right|}{\s^{2}} \le C\phi^{\circ}_{{\cal N}(0,1)}(s,d).
	\end{equation*}

\item[(ii)] 
Let $s$ and $d$ be integers satisfying $1\le s\le d$ and
let $\ell(\cdot)$ be any loss function in the class $\mathcal L$.  Then, 
	\begin{equation*}\label{lowerbound:variance:gauss}
	\inf_{\hat{T}}  \sup_{\s>0} \sup_{\|\bt\|_0\le s} 
	\esp_{\bt,{\cal N}(0,1),\s}\,\ell \bigg( c(\phi^{\circ}_{{\cal N}(0,1)}(s,d))^{-1} 
\bigg|\frac{ \hat{T}}{\s^{2}}-1\bigg| \bigg)\ge c^{\prime},
	\end{equation*}
	where  $\inf_{\hat{T}}$ denotes the infimum over all estimators, and $c>0$, $c^{\prime}>0$ are constants that can depend only on $\ell(\cdot)$.
	}
	\end{itemize}

\end{proposition}
}

Estimators of variance or covariance matrix based on the empirical characteristic function have been 
studied in several papers \cite{ButuceaMatias2005,CaiJin2010,BelomestnyTrabsTsybakov2017,CarpentierVerzelen2019}. 
The setting in \cite{ButuceaMatias2005,CaiJin2010,BelomestnyTrabsTsybakov2017} is different from ours as 
those papers deal with the model where the non-zero components of $\bt$ are random with a 
smooth distribution density. The estimators in \cite{ButuceaMatias2005,CaiJin2010} are also quite different. On the other hand, 
\cite{BelomestnyTrabsTsybakov2017,CarpentierVerzelen2019} consider estimators
close to $\tilde{v}^{2}$. In particular, \cite{CarpentierVerzelen2019} uses a similar pilot 
estimator for testing in the sparse vector model where it is assumed 
that $\s\in [\s_{-},\s_+]$,  
$0<\s_{-}<\s_+<\infty$, and the estimator depends on $\s_+$. 
Although \cite{CarpentierVerzelen2019} 
does not provide explicitly stated result about the rate of this estimator, the   
proofs in \cite{CarpentierVerzelen2019} come close to it and we 
believe that it satisfies an upper bound as in item (i) of Proposition 
\ref{upperbound:variance:gauss} with $\sup_{\s>0}$ replaced by 
$\sup_{\s\in[\s_{-},\s_+]}$.


%
%
%
%

\subsection{Distribution-free variance estimators}
\label{sec:free}

The main drawback of the estimator $\hat{\s}_{\sf med}^2$ is the dependence on the parameter $\b$. 
It  reflects the fact that the estimator is tailored for a given and known distribution of noise $F$.
Furthermore, as shown below,  the rate~\eqref{rate:gao} cannot be achieved if it is only known that $F$ 
belongs to one of the classes of distributions that we consider in this paper.

Instead of using one particular quantile, like the median in Section \ref{sec:median}, one can estimate $\s^2$ by an integral over all quantiles, which allows one to avoid considering distribution-dependent quantities like~($\ref{equation_quantile}$).

Indeed, with the notation $q_\a=G^{-1}(1-\a)$ where $G$ is the c.d.f. of $(\s\xi_1)^2$ and $0<\a<1$, the variance of the noise can be expressed as
\begin{equation*}
	\s^2 = \esp(\s\xi_1)^2 = \int_0^1 q_\a\,d\a.
\end{equation*}
Discarding the higher order quantiles that are dubious in the presence of outliers and replacing $q_\a$ by the empirical quantile $\hat{q}_\a$ of level $\a$ we obtain the following estimator
\begin{equation}\label{definition_noisevarianceestimator}
	\hat{\s}^2 = \int_0^{1-s/d} \hat{q}_\a\,d\a = \frac1d \sum_{k=1}^{d-s} Y^2_{(k)},
\end{equation}
where  $Y^2_{(1)}\le\ldots\le Y^2_{(d)}$ are the ordered values of the squared observations $Y_1^2,\dots, Y_d^2$.  
Note that $\hat{\s}^2$ is an $L$-estimator, \cf\cite{Huber1981}.  Also, up to a constant factor, $\hat{\s}^2$ coincides with the statistic used in~\cite{CollierCommingesTsybakov2017} .

The following theorem provides an upper bound on the risk  of the estimator $\hat{\s}^2$ under the assumption that the noise belongs to the class~$\mathcal{G}_{a,\tau}$.
Set
$$
\phi_{\sf exp}(s,d) = \max\left(\frac1{\sqrt{d}}, \frac{s}{d}\log^{2/a}\left(\frac{ed}{s}\right)\right). 
$$

\begin{theorem}\label{theorem_upperbound_noise_subgaussian}
	Let $\tau>0$, $a>0$, and let $s,d$ be integers satisfying $1\le s < d/2$. Then,    the estimator $\hat{\s}^2$ defined in~(\ref{definition_noisevarianceestimator}) satisfies
	\begin{equation}
		 \sup_{P_\xi\in\mathcal{G}_{a,\tau}} \sup_{\s>0} \sup_{\|\bt\|_0\le s} 
\frac{\esp_{\bt,P_\xi,\s} \big(\hat{\s}^2-\s^2\big)^2}{\s^4} \le C \phi_{\sf exp}^2(s,d) 
	\end{equation}
	where $C>0$ is a constant depending only on $a$ and $\tau$. 
\end{theorem}

The next theorem establishes the performance of variance estimation in the case of distributions with polynomially decaying tails.
Set
$$
\phi_{\sf pol}(s,d) =  \max\left(\frac1{\sqrt{d}}, \Big(\frac{s}{d}\Big)^{1-\frac{2}{a}} \right). 
$$

\begin{theorem}\label{theorem_upperbound_noise_polynomial}
	Let $\tau>0,a>4$, and let $s,d$ be integers satisfying $1\le s< d/2$. Then,    the estimator $\hat{\s}^2$ defined in~(\ref{definition_noisevarianceestimator}) satisfies	
	\begin{equation}
		\sup_{P_\xi\in\mathcal{P}_{a,\tau}} \sup_{\s>0} 
\sup_{\|\bt\|_0\le s} \frac{\esp_{\bt,P_\xi,\s} \big(\hat{\s}^2-\s^2\big)^2}{\s^4} \le C \phi_{\sf pol}^2(s,d),
	\end{equation}
	where $C>0$ is a constant depending only on $a$ and $\tau$. 
\end{theorem}

We assume here that the noise distribution has  a moment of order greater than 4, which is close to the minimum requirement since we deal with the expected squared error of a quadratic function of the observations.

We now state the lower bounds matching the results of 
Theorems~\ref{theorem_upperbound_noise_subgaussian} and \ref{theorem_upperbound_noise_polynomial}.
\begin{theorem}\label{theorem_lowerbound_noise_subgaussian}
	Let $\tau>0$, $a>0$, and let $s,d$ be integers satisfying $1\le s\le d$. Let $\ell(\cdot)$ be any loss function in the class $\mathcal L$. Then, 
		\begin{equation}
		\inf_{\hat{T}} \sup_{P_\xi\in\mathcal{G}_{a,\tau}} \sup_{\s>0} \sup_{\|\bt\|_0\le s}\esp_{\bt,P_\xi,\s}\,\ell \Big( c(\phi_{\sf exp}(s,d))^{-1} \Big| \frac{ \hat{T}}{\s^2} -1\Big| \Big)\ge c',
	\end{equation}
where $\inf_{\hat{T}}$ denotes the infimum over all estimators and $c>0$, $c'>0$ are constants depending only on $\ell(\cdot)$, $a$ and $\tau$.
\end{theorem}

Theorems~\ref{theorem_upperbound_noise_subgaussian} and~\ref{theorem_lowerbound_noise_subgaussian} imply that the estimator $\hat{\s}^2$ is rate optimal in a minimax sense when the noise belongs to $\mathcal{G}_{a,\tau}$, in particular when it is sub-Gaussian.
Interestingly, an extra logarithmic factor appears in the optimal rate when passing from the pure 
Gaussian distribution of $\xi_i$'s (\cf Proposition~\ref{prop:variance:gauss}) to the class of all sub-Gaussian distributions. This factor can be seen as a price to pay for the lack of information regarding the exact form of the distribution.

Under polynomial tail assumption on the noise, we have the following minimax lower bound.

\begin{theorem}\label{theorem_lowerbound_noise_polynomial}
Let $\tau>0$, $a \ge 2$, and let $s,d$ be integers satisfying $1\le s\le d$. Let $\ell(\cdot)$ be any loss function in the class $\mathcal L$. Then, 	
	\begin{equation}
	\inf_{\hat{T}}  \sup_{P_\xi\in\mathcal{P}_{a,\tau}} \sup_{\s>0} \sup_{\|\bt\|_0\le s} 
	\esp_{\bt,P_\xi,\s}\,\ell \Big( c(\phi_{\sf pol}(s,d))^{-1} \Big| \frac{ \hat{T}}{\s^2} -1\Big| \Big)\ge c'
	\end{equation}
where $\inf_{\hat{T}}$ denotes the infimum over all estimators and $c>0$, $c'>0$ are constants depending only on $\ell(\cdot)$, $a$ and $\tau$.
\end{theorem}

This theorem shows that the rate $\phi_{\sf pol}(s,d)$ obtained in Theorem~\ref{theorem_upperbound_noise_polynomial} cannot be improved in a minimax sense. 

A drawback of the estimator defined in~(\ref{definition_noisevarianceestimator}) is in the lack of adaptivity  to the sparsity parameter $s$. At first sight, it may seem that the estimator 
\begin{equation}\label{naive_estimator}
	\hat{\s}_*^2 
	= \frac2d \sum_{1\le k \le d/2} Y^2_{(k)}
\end{equation}
could be taken as its adaptive version. However, $\hat{\s}_*^2$ is not a good estimator of $\s^2$ as can be seen from the following proposition.

\begin{proposition}\label{proposition_suboptimality}
	Define $\hat{\s}_*^2$ as in~(\ref{naive_estimator}). Let $\tau>0$, $a\ge 2$, and let $s,d$ be integers satisfying $1\le s\le d$, and $d=4k$ for an integer $k$. Then,
	\begin{equation*}
		\sup_{P_\xi\in\mathcal{G}_{a,\tau}} \sup_{\s>0} \sup_{\|\bt\|_0\le s} \frac{\esp_{\bt,P_\xi,\s} \big(\hat{\s}_*^2-\s^2\big)^2}{\s^4} \ge \frac{1}{64}.
	\end{equation*}
\end{proposition}

On the other hand, it turns out that a simple plug-in estimator 
\begin{equation}\label{definition_adaptive_variance}
	\hat{\s}^2 = \frac1d \|\bY-\hat{\bt}\|_2^2
\end{equation}
with $\hat{\bt}$ chosen as in Section \ref{sec:vector} achieves rate optimality adaptively  to 
the  noise distribution and to the sparsity parameter $s$. 
This is detailed 
in the next theorem.

\begin{theorem}\label{theorem_adaptiveupperbound_variance}
{
	 Let $s$ and $d$ be integers satisfying $1\le s< \lfloor \g d\rfloor/4$, where 
$\g\in(0,1/2]$ is the tuning parameter in the definition of $\tilde \s^2$. 
Let $\hat{\s}^2$  be the estimator defined by~\eqref{definition_adaptive_variance} where 
$\hat{\bt}$ is defined in~\eqref{def_estimateur_mom}. Then the following properties hold.
	\begin{enumerate}
		\item 
		Let $\tau>0, a>0$. There exist constants
	 $c,C>0$ and $\g\in(0,1/2]$ depending only on $(a,\tau)$ such that if $\l_j=c \log^{1/a}(ed/j), j=1,\dots,d $, we have
	\begin{align*}
				\sup_{ P_\xi\in\mathcal{P}_{a,\tau}}\sup_{\s>0}\sup_{\|\bt\|_0\le s} \frac{\esp_{\bt,P_\xi,\s} \big|\hat{\s}^2-\s^2 \big|}{\s^{2}}
				\le C \phi_{\sf exp}(s,d) .
			\end{align*}
		\item Let $\tau>0, a> 4$. There exist constants
	 $c,C>0$ and $\g\in(0,1/2]$ depending only on $(a,\tau)$ such that if 
$\l_j=c ({d}/{j})^{1/a}, j=1,\dots,d$, we have
	\begin{align*}
				\sup_{ P_\xi\in\mathcal{P}_{a,\tau}}\sup_{\s>0}\sup_{\|\bt\|_0\le s} \frac{\esp_{\bt,P_\xi,\s} \big|\hat{\s}^2-\s^2 \big|}{\s^{2}} 
				\le C \phi_{\sf pol}(s,d) .
			\end{align*}
	\end{enumerate}
}
\end{theorem}


\section{Proofs of the upper bounds}

\subsection{\it Proof of Proposition~\ref{proposition_over}}
 Fix $\bt \in \Theta_s$ and let $S$ be the support of $\bt$. We will call outliers the observations $Y_i$ with $i\in S$. There are at least $m-s$ blocks $B_i$ that do not contain outliers. 
{
Denote by $J$ a set of $m-s$ indices $i$, for which $B_i$ contains no outliers.} 

As $a>2$, there exist  constants $L=L(a,\tau)$ and $r=r(a,\tau)\in (1,2]$  such that 
$\esp| \xi_1^2-1|^r\le L$.  Using von Bahr-Esseen inequality (\cf\cite{petrov}) 
and the fact that $|B_i|\ge k$ we get  
$$\prob\Big( \Big|\frac{1}{|B_i|}\sum_{j\in B_i}  \xi_j^2-1 \Big| > 1/2\Big)
\le \frac{2^{r+1}L}{ k^{r-1}} , \quad i=1,\dots,m.$$
Hence,  there exists a constant $C_1=C_1(a,\tau)$  
such that if $k\ge C_1$ (i.e., if $\g$ is small enough depending on $a$ and $\tau$), then 
\begin{align}\label{pprob}
\prob_{\bt,P_\xi,\s}(\bar{\s}_i^2\notin I)\le \frac{1}{4}, \quad i=1,\dots,m,
\end{align}
where $I=[\frac{\s^2}{2}, \frac{3\s^2}{2}]$.
Next, by the definition of the median, for any interval $I\subseteq \RR$ we have 
\begin{align}
\prob_{\bt,P_\xi,\s}(\tilde{\s}^2\notin I)&\le\prob_{\bt,P_\xi,\s}\Big( \sum_{i=1}^{m} \fcar_{\bar{\s}_i^2 \notin I} \ge \frac{m}{2}  \Big)\le \prob_{\bt,P_\xi,\s}\Big( \sum_{i\in J} \fcar_{\bar{\s}_i^2 \notin I} \ge \frac{m}{2}-s  \Big).
\end{align}
Now, $s\le \frac{\lfloor \g d \rfloor}{4}=\frac{m}{4}$, so that $\frac{m}{2}-s\ge \frac{m-s}{3}$. 
Set $\eta_i= \fcar_{\bar{\s}_i^2 \notin I}$, $i\in J$. Due to \eqref{pprob} we have 
$\esp(\eta_i)\le 1/4$, and $(\eta_i, i\in J)$ are independent.  Using these remarks and Hoeffding's inequality we find 
$$
\prob\Big( \sum_{i\in J} \eta_i \ge \frac{m}{2}-s  \Big)\le \prob\Big( \sum_{i\in J} (\eta_i - \esp(\eta_i))\ge \frac{m-s}{12}\Big) \le \exp(-C (m-s)).
$$
Note that $|J|=m-s\ge 3m/4=3{\lfloor \g d \rfloor}/4$.
 Thus, if $\gamma$ is chosen small enough depending only on $a$ and $\tau$ then 
$$\prob_{\bt,P_\xi,\s}(\tilde{\s}^2\notin I )\le \exp(- C d).$$
{
This proves the desired bound in probability. To obtain the bounds in expectation, set $Z=\left|\tilde{\s}^{2} - \s^{2}\right|$. 
Let first $a>4$ and take some $r\in (1, a/4)$. Then 
\begin{align*}
\esp_{\bt,P_\xi,\s}\left( Z^{2}  \right) & \leq \frac{\s^{4}}{4} + \esp_{\bt,P_\xi,\s}\left( Z^2 \fcar_{Z \ge \frac{\s^{2}}{2}} \right) \\
& \le \frac{9\s^{4}}{4} + 2\left(\esp_{\bt,P_\xi,\s}\left( \tilde{\s}^{4r} \right)\right)^{1/r} 
\left(\prob_{\bt,P_\xi,\s}\left( Z \ge {\s^{2}}/{2} \right)\right)^{1-1/r} \\
& \le \frac{9\s^{4}}{4} + 2\left(\esp_{\bt,P_\xi,\s}\left( \tilde{\s}^{4r} \right)\right)^{1/r}
\exp(- C d).
\end{align*}
Since $m \ge 4s$, we can easily argue that $ \tilde{\s}^{4r} \leq \sum_{i\in J}\bar{\s}_{i}^{4r} $. It follows that
$$
\esp_{\bt,P_\xi,\s}\left( \tilde{\s}^{4r} \right) \le C\s^{4r}d^{2}.
$$
Hence 
$
\esp_{\bt,P_\xi,\s}\left( Z^{2}  \right) \leq C\s^{4}.
$
Similarly, if $a>2$, then
$
\esp_{\bt,P_\xi,\s}\left( Z  \right) \leq C\s^{2}.
$
}

\subsection{Proof of Theorem~\ref{theorem_adaptiveupperbound}}

Set $\bu=\hat{\bt}-\bt$. It follows from Lemma A.2 in \cite{{BellecLecueTsybakov2017}} that   
$$
	2\|\bu\|_2^2\le 2 \s \sum_{i=1}^d \xi_i u_i+\tilde{\s}\| \bt\|_*-\tilde{\s}\| \hat{\bt}\|_*,$$
where $u_i$ are the components of $\bu$. Next, Lemma A.1 in \cite{{BellecLecueTsybakov2017}} yields 
	\begin{equation*}\label{slope_pol_2}	
	\| \bt\|_*-\| \hat{\bt}\|_*\le \Big(\sum_{j=1}^s \l_j^2\Big)^{1/2} \| \bu\|_2 -\sum_{j=s+1}^d \l_j |u|_{(d-j+1)}
	\end{equation*}
	where $|u|_{(k)}$ is the $k$th order statistic of $|u_1|,\dots,|u_d|$.
	Combining these two inequalities we get
	\begin{equation}\label{combination}
	2\|\bu\|_2^2\le 2 \s \sum_{j=1}^d \xi_j u_j+\tilde{\s}\Big\{ \Big(\sum_{j=1}^s \l_j^2\Big)^{1/2} \| \bu\|_2 -\sum_{j=s+1}^d \l_j |u|_{(d-j+1)} \Big\}.
	\end{equation}
	{
	 For some permutation $(\varphi(1), \dots, \varphi(d))$ of $(1,\ldots,d)$, we have 
\begin{equation}\label{permutation}
 \Big|\sum_{i=1}^d \xi_j u_j \Big| \le  \sum_{j=1}^d |\xi|_{(d-j+1)} |u_{\varphi(j)}| \le  \sum_{j=1}^d |\xi|_{(d-j+1)} |u|_{(d-j+1)},
\end{equation}
where the last inequality is due to the fact that the sequence $|\xi|_{(d-j+1)}$ is non-increasing. Hence
\begin{align*}
	2\|\bu\|_2^2 &\le 2 \s \sum_{j=1}^s |\xi|_{(d-j+1)} |u|_{(d-j+1)}+\tilde{\s} \Big(\sum_{j=1}^s \l_j^2\Big)^{1/2} \| \bu\|_2 + \sum_{j=s+1}^d \left(2\s |\xi|_{(d-j+1)} - \tilde{\s}\l_{j} \right)|u|_{(d-j+1)}\\
	&\le \left\{ 2 \s\Big(\sum_{j=1}^s |\xi|_{(d-j+1)} ^2\Big)^{1/2} + \tilde{\s}\Big(\sum_{j=1}^s \l_j^2\Big)^{1/2}  + \Big(\sum_{j=s+1}^d \left(2\s |\xi|_{(d-j+1)} - \tilde{\s}\l_{j} \right)^{2}_{+} \Big)^{1/2}\right\} \| \bu\|_2.
\end{align*}
This implies
		$$
		\|\bu\|^{2}_2 \le C \left\{  \s^{2}\sum_{j=1}^s |\xi|_{(d-j+1)} ^2 + \tilde{\s}^{2}\sum_{j=1}^s \l_j^2 + \sum_{j=s+1}^d \left(2\s |\xi|_{(d-j+1)} - \tilde{\s}\l_{j} \right)^{2}_{+} \right\}.
		$$
From Lemmas \ref{lemma_esp4_subgaussian} and \ref{lemma_esp4_polynomial} we have 
$\esp(|\xi|_{(d-j+1)} ^2)\le C\l_j^2$.	Using this and Proposition \ref{proposition_over} we obtain 
	    \begin{equation}\label{eqq}
		\esp_{\bt,P_\xi,\s}\left(\|\bu\|^{2}_2\right) \le C \left( 
\s^{2}\sum_{j=1}^s \l_j^2 + \esp_{\bt,P_\xi,\s}\Bigg(\sum_{j=s+1}^d \left(2\s |\xi|_{(d-j+1)} - 
\tilde{\s}\l_{j} \right)^{2}_{+}\Bigg)\right) .
		\end{equation}
Define the events $\mathcal{A}_{j}=\Big\{   |\xi|_{(d-j+1)}\le {\l_j}/{4} \Big\}\cap \Big\{ 1/2 \le {\tilde{\s}^2}/{\s^2}\le 3/2\Big\}$ for $j=s+1,\ldots,d$. Then
		$$
		\esp_{\bt,P_\xi,\s}\left(\sum_{j=s+1}^d \left(2\s |\xi|_{(d-j+1)} - \tilde{\s}\l_{j} \right)^{2}_{+}\right) \le 4\s^{2}\esp_{\bt,P_\xi,\s}\left(\sum_{j=s+1}^d|\xi|^{2}_{(d-j+1)}\fcar_{\mathcal{A}_{j}^{c}}  \right).
		$$
Fixing some $1<r<a/2 $ we get 
		$$
		\esp_{\bt,P_\xi,\s}\left(\sum_{j=s+1}^d \left(2\s |\xi|_{(d-j+1)} - \tilde{\s}\l_{j} \right)^{2}_{+}\right) \le 4\s^{2}\sum_{j=s+1}^d\esp\left(|\xi|^{2r}_{(d-j+1)}\right)^{1/r}\prob_{\bt,P_\xi,\s}\left(\mathcal{A}_{j}^{c}\right)^{1-1/r}.
		$$
Lemmas~\ref{lemma_esp4_subgaussian}, \ref{lemma_esp4_polynomial} and the definitions of parameters $\l_j$
imply that
		$$
		 \esp\left(|\xi|^{2r}_{(d-j+1)}\right)^{1/r} \le C\l_{s}^2, \quad j=s+1,\dots,d.
		$$
Furthermore, it follows from the proofs of Lemmas~\ref{lemma_esp4_subgaussian} and \ref{lemma_esp4_polynomial} that if
the constant $c$ in the definition of $\l_{j}$ is chosen large enough, then
$\prob (|\xi|_{(d-j+1)} >\l_{j}/4)  \le  q^{j}$ for some $q<1/2$ depending only on $a$ and $\tau$. 
This and Proposition \ref{proposition_over} imply that		
		$
		\prob_{\bt,P_\xi,\s}(\mathcal{A}_{j}^{c}) \le e^{-cd} + q^{j}.
		$
		Hence, 
		$$
		\esp_{\bt,P_\xi,\s}\left(\sum_{j=s+1}^d \left(2\s |\xi|_{(d-j+1)} - \tilde{\s}\l_{j} \right)^{2}_{+}\right) \le C\s^{2}\l^{2}_{s}\sum_{j=s+1}^d(e^{-cd} + q^{j})^{1-1/r} \le C^{\prime} \s^{2}\sum_{j=1}^s \l_j^2.
		$$
Combining this inequality with \eqref{eqq} we obtain
\begin{equation}\label{eqq1}
		\esp_{\bt,P_\xi,\s}\left(\|\bu\|^{2}_2\right) \le C \s^{2}\sum_{j=1}^s \l_j^2.
		\end{equation}
To complete the proof, it remains to note that $\sum_{j=1}^s \l_j^2\le C (\phi_{\sf pol}^*(s,d))^2$ in the polynomial case
and $\sum_{j=1}^s \l_j^2\le C (\phi_{\sf exp}^*(s,d))^2$ in the exponential case, cf. Lemma \ref{lemma:sum}.	
	
\subsection{Proof of part (i) of Proposition~\ref{prop:norm:gauss}} 
We consider separately the "dense" zone $s>\sqrt{d}$ and the "sparse" zone $s\le\sqrt{d}$. Let first $s>\sqrt{d}$. Then the rate $\phi_{{\cal N}(0,1)}^*(s,d)$ is of order $\sqrt{\frac{s}{1+\log_+(s^{2}/d)}}$.  Thus, for $s>\sqrt{d}$ we need to prove that 
\begin{equation}\label{g1}
\sup_{\s>0}\sup_{\|\bt\|_0\le s} \esp_{\bt, {\cal N}(0,1),\s}  
\bigg(  \bigg|\frac{ \hat{N}^*-\|\bt\|_2}{\s}\bigg|^{2} \bigg)\le  \frac{Cs}{1+\log_+(s^{2}/d)}.
	\end{equation}
	Denoting $\bxi=(\xi_1,\dots,\xi_d)$ we have
	\begin{eqnarray}\label{g2}
	\Big| \hat{N}^* - \|\bt\|_2\Big| &=& 
\bigg| \Big|  \sum_{j=1}^d Y_j^2 -d\hat{\s}^2\Big|^{1/2} - \|\bt\|_2\bigg|\\ \nonumber
&=&\bigg| \sqrt{ \big| \|\bt\|_2^2+2\s(\bt,\bxi) + \s^2\|\bxi\|_2^2 -d\hat{\s}^2\big|} - \|\bt\|_2\bigg|
\\  \nonumber
&\le&\bigg| \sqrt{ \big| \|\bt\|_2^2+2\s(\bt,\bxi) \big|} - \|\bt\|_2\bigg| +  \s \sqrt{\big|\|\bxi\|_2^2 -d\big| } 
+\sqrt{d |\s^2 - \hat{\s}^2 |}.
	\end{eqnarray}
	The first term in the last line vanishes if $\bt= 0$, while for $\bt\ne 0$ it is bounded as follows:
	\begin{equation}\label{g3}
	\bigg| \sqrt{ \big| \|\bt\|_2^2+2\s(\bt,\bxi) \big|} - \|\bt\|_2\bigg| = \|\bt\|_2\bigg| \sqrt{ \Big| 1+\frac{2\s(\bt,\bxi)}{\|\bt\|_2^2} \Big|} - 1\bigg|\le \frac{2\s|(\bt,\bxi)|}{\|\bt\|_2}
	\end{equation}
	where we have used the inequality $| \sqrt{ | 1+x|} - 1|\le  |x|$, $\forall x\in \RR$. Since here $|(\bt,\bxi)|/\|\bt\|_2\sim {\cal N}(0,1)$    we have, for all $\bt$,
	\begin{equation}\label{g31}
\esp\left( \left| \sqrt{ \big| \|\bt\|_2^2+2\s(\bt,\bxi) \big|} - \|\bt\|_2\right|^{2} \right) \le 4 \s^{2},
	\end{equation}
	and since $ \|\bxi\|_2^2$ has a chi-square distribution with $d$ degrees of freedom we have
	\begin{eqnarray}\label{g4}
	\esp\Big( \big| \|\bxi\|_2^2-d \big| \Big)&\le &
 \left(\esp\Big( \big| \|\bxi\|_2^2-d \big|^{2} \Big)\right)^{1/2}  =  \sqrt{2d}.\nonumber
\end{eqnarray}
Next, by Proposition~\ref{prop:variance:gauss} we have that, for $s> \sqrt{d}$,
\begin{equation}\label{gf1}
		\sup_{\s>0}\sup_{\|\bt\|_0\le s} \esp_{\bt, {\cal N}(0,1),\s} 
\left( \Big|\frac{\hat{\s}^2}{\s^2}-1\Big| \right) \le \frac{Cs}{d(1+\log_+(s^{2}/d))}
	\end{equation}
	for some absolute constant $C>0$. Combining \eqref{g2} -- \eqref{g4} yields \eqref{g1}.
	
	Let now $s\le \sqrt{d}$. Then the rate $\phi_{{\cal N}(0,1)}^*(s,d)$ is of 
order $\sqrt{s\log(1+d/s^{2})}$.  Thus, for $s \le \sqrt{d}$ we need to prove that 
\begin{equation}\label{gf10}
\sup_{\s>0}\sup_{\|\bt\|_0\le s} \esp_{\bt, {\cal N}(0,1),\s}  \bigg(  \bigg|\frac{ \hat{N}^*-\|\bt\|_2}{\s}\bigg|^{2} \bigg)\le C s\log(1+d/s^{2}).
	\end{equation}
We have
	\begin{eqnarray}\label{gf2}
	\quad \Big| \hat{N}^* - \|\bt\|_2\Big| &=& 
    \bigg|  \Big| \sum_{j=1}^d (Y_j^2~\fcar_{\{ |Y_j|>\rho_{j} \}}) -d\a \hat{\s}^2\Big|^{1/2} - \|\bt\|_2\bigg|\\ \nonumber
    &=& 
    \bigg|  \Big| \sum_{j \in S}(Y_j^2~\fcar_{\{ |Y_j|>\rho_{j} \}}) + \s^{2}\sum_{j \not\in S}(\xi_{j}^2~\fcar_{\{ \s|\xi_j|>\rho_{j} \}}) -  d\a \hat{\s}^2\Big|^{1/2} - \|\bt\|_2\bigg|\\ \nonumber
    &\le&\left| \sqrt{ \sum_{j \in S}(Y_j^2~\fcar_{\{ |Y_j|>\rho_{j} \}})} - \|\bt\|_2 \right| + \left| \s^{2}\sum_{j \not\in S}(\xi_{j}^2~\fcar_{\{ \s|\xi_j|>\rho_{j} \}}) -  d\a \hat{\s}^2\right|^{1/2}.
    	\end{eqnarray}
Here,
\begin{eqnarray}\label{gf3}
	\left| \sqrt{ \sum_{j \in S}(Y_j^2~\fcar_{\{ |Y_j|>\rho_{j} \}})} - \|\bt\|_2 \right| &\leq& \left| \sqrt{ \sum_{j \in S}(Y_j~\fcar_{\{ |Y_j|>\rho_{j} \}} - \theta_{j})^{2}} \right| \\ \nonumber
	&\leq& \sqrt{\sum_{j\in S}\rho_{j}^2} + \s \sqrt{ \sum_{j \in S} \xi_{j}^{2}}.
    	\end{eqnarray}
Hence, writing for brevity $\esp_{\bt, {\cal N}(0,1),\s}=\esp$, we get 
\begin{eqnarray}\label{gf3a}\nonumber
\esp \left(  \left| 
\sqrt{ \sum_{j \in S}(Y_j^2~\fcar_{\{ |Y_j|>\rho_{j} \}})} - \|\bt\|_2 \right|^{2}\right) &\leq&
 16 \esp\left(\hat{\s}_{\sf med, 1}^{2}+\hat{\s}_{\sf med,2}^{2}\right) s \log\big(1+{d}/{s^{2}}\big) + 2\s^{2}s
\\
& \leq& C\s^{2}s\log(1+d/s^{2}),\nonumber
\end{eqnarray}
where we have used the fact that 
$\esp\left(|\hat{\s}_{\sf med, k}^{2} - \s^{2}| \right) \leq C \s^{2}$, 
${\sf k}=1,2$, by Proposition~\ref{prop:gao}.
Next, we study the term
$\Gamma= \left| \s^{2}\sum_{j \not\in S}(\xi_{j}^2~\fcar_{\{ \s|\xi_j|>\rho_{j} \}}) -  
d\a \hat{\s}^2\right| $. We first write
\begin{eqnarray}\label{gf4}
	 \Gamma &\leq& \left| \s^{2}\sum_{j \not\in S}\xi_{j}^2(~\fcar_{\{ \s|\xi_j|>\rho_{j} \}} 
- ~\fcar_{\{ \s|\xi_j|>t_{*} \}}) \right| + \left|\s^{2}\sum_{j \not\in S}(\xi_{j}^2 ~\fcar_{\{ \s|\xi_j|>t_{*} \}}) - d\a\hat{\s}^2 \right|,
    	\end{eqnarray}
where $t_{*}=  2\s  \sqrt{2\log (1+d/s^2)}$. For the second summand on the right hand side of \eqref{gf4} 
we have
$$
\left|\s^{2}\sum_{j \not\in S}(\xi_{j}^2 ~\fcar_{\{ \s|\xi_j|>t_{*} \}}) - d\a\hat{\s}^2 \right| \leq \s^{2}\left|\sum_{j \not\in S}(\xi_{j}^2 ~\fcar_{\{ \s|\xi_j|>t_{*} \}}) - (d-|S|)\a \right| 
+ \left| \s^{2}-\hat{\s}^2 \right|d\a+|S|\a\s^{2},
$$
where $|S|$ denotes the cardinality of $S$. By Proposition~\ref{prop:variance:gauss}  
 we have $\esp( |\hat{\s}^2 - \s^2|) \le C\s^2/\sqrt{d}$ for $s\le \sqrt{d}$. Hence,
$$
\esp\left|\s^{2}\sum_{j \not\in S}(\xi_{j}^2 ~\fcar_{\{ \s|\xi_j|>t_{*} \}}) - d\a\hat{\s}^2 \right| \leq \s^{2}\sqrt{d\esp\left(\xi_1^4~\fcar_{\{ |\xi_1|>  \sqrt{2\log (1+d/s^2)} \}}\right)} + C\a\s^{2}\left( \sqrt{d} + s\right).
$$
It is not hard to check (cf., e.g., \cite[Lemma 4]{CollierCommingesTsybakov2017}) that, for $s\le \sqrt{d}$,
$$\alpha \leq C (\log\left(1+d/s^{2} \right))^{1/2}\frac{s^{2}}{d},$$
and  
$$
\esp\left(\xi_1^4~\fcar_{\{ |\xi_1|>  \sqrt{2\log (1+d/s^2)} \}}\right) 
\leq C (\log\left(1+d/s^{2}\right))^{3/2}\frac{s^{2}}{d}, 
$$
so that
$$
\esp\left|\s^{2}\sum_{j \not\in S}(\xi_{j}^2 ~\fcar_{\{ \s|\xi_j|>t_{*} \}}) - d\a\hat{\s}^2 \right| 
\leq C\s^{2}s\log (1+d/s^2).
$$
Thus, to complete the proof it remains to show that
\begin{equation}\label{eqq2}
\s^{2} \sum_{j \not\in S}\esp\left|\xi_{j}^2(\fcar_{\{ \s|\xi_j|>\rho_{j} \}} - 
\fcar_{\{ \s|\xi_j|>t_{*} \}}) \right| \le C\s^{2}s\log (1+d/s^2).
\end{equation}
Recall that  $\rho_{j}$ is independent from $\xi_{j}$. Hence, conditioning on $\rho_{j}$ we obtain  
\begin{align}\label{eqq4}
\s^{2}\esp\left(\left|\xi_{j}^2(\fcar_{\{ \s|\xi_j|>\rho_{j} \}} - \fcar_{\{ \s|\xi_j|>t_{*} \}}) 
\right|\rho_{j}\right) \le |\rho_{j}^{2}-t_{*}^2|e^{-t_{*}^2/(8\s^{2})} + \s^{2}\fcar_{\{\rho_{j}< t_{*}/2\}},
\end{align}
where we have used the fact that, for $b>a>0$,
$$
\int_a^b x^2 e^{-x^2/2}dx \le \int_a^b x e^{-x^2/4}dx\le |b^2 - a^2| e^{-\min (a^2, b^2)/4}/2.
$$
Using Proposition~\ref{prop:gao} and
definitions  of $\rho_{j}$ and $t_{*}$, we get that, for $s\le \sqrt{d}$,
\begin{align}\label{eqq5}
\esp\left(|\rho_{j}^{2}-t_{*}^2|\right) e^{-t_{*}^2/(8\s^{2})}  
&\le 8 \max_{{\sf k}=1,2}\esp( |\hat{\s}^{2}_{\sf med,k} - \s^{2}|) \frac{s^2}{d}\log(1+d/s^{2})
 \\ \nonumber
&\le C\s^{2} \frac{s}{d}\log(1+d/s^{2}).
\end{align}
Next, it follows from  Proposition~\ref{prop:gao} that there exists $\gamma\in (0,1/8)$ small enough 
such that for $s\le \g d$ we have 
$\max_{{\sf k}=1,2}\prob( \hat{\s}^{2}_{\sf med,k} < \s^{2}/2)\le 2 e^{-c_\g d}$
where $c_\g>0$ is a constant. Thus, 
$
\s^{2}\prob(\rho_{j} < t_{*}/2) \le 2 \s^{2} e^{-c_\g d} \le C\s^{2}(s/d)\log (1+d/s^2).
$
Combining this with \eqref{eqq4} and \eqref{eqq5} proves \eqref{eqq2}. 

\subsection{Proof of part (i) of Proposition~\ref{prop:norm:known_sigma} and part (i) of Proposition~\ref{prop:norm:poly:known_sigma}}
 We only prove Proposition \ref{prop:norm:known_sigma} since 
the proof of Proposition \ref{prop:norm:poly:known_sigma} is similar taking into account that 
$\mathbf{E}(\xi_{1}^{4})<\infty$.  We consider separately the "dense" zone 
$s>\frac{\sqrt{d}}{\log^{\frac{2}{a}}(ed)}$  and the 
"sparse" zone $s \le \frac{\sqrt{d}}{\log^{\frac{2}{a}}(ed)}$ . Let first $s>\frac{\sqrt{d}}{\log^{\frac{2}{a}}(ed)}$ . Then the rate $\phi_{\sf exp}^\circ(s,d)$ is of order $d^{1/4}$ and thus
  we need to prove that 
\begin{equation*}
\sup_{P_\xi \in \mathcal{G}_{a,\tau}}\sup_{\|\bt\|_0\le s} 
\esp_{\bt, P_\xi,\s}  \big(  | \hat{N}_{\sf exp}^{\circ} - \|\bt\|_2 |^{2} \big)\le C\sigma^{2}\sqrt{d}.
	\end{equation*}
Since $\s$ is known, arguing similarly to \eqref{g2} - \eqref{g3} we find
$$
| \hat{N}_{\sf exp}^{\circ} - \|\bt\|_2| \le \left|\frac{2\s|(\bt,\bxi)|}{\|\bt\|_2}\right|\fcar_{\bt \neq 0} +  \s \sqrt{\big|\|\bxi\|_2^2 -d\big| }.
$$
As $\mathbf{E}(\xi_{1}^{4})<\infty$, this implies 
$$
\esp_{\bt, P_\xi,\s}  \big(  | \hat{N}_{\sf exp}^{\circ} - \|\bt\|_2 |^{2} \big)
\le 8 \s^{2} + C\s^{2}\sqrt{d},
$$
which proves the result in the dense case.
Next, in the sparse cas $s \leq \frac{\sqrt{d}}{\log^{\frac{2}{a}}(ed)}$, we need to prove that 
\begin{equation*}
\sup_{P_\xi \in \mathcal{G}_{a,\tau}}\sup_{\|\bt\|_0\le s} \esp_{\bt, P_\xi,\s}  
\big(  | \hat{N}_{\sf exp}^{\circ} - \|\bt\|_2 |^{2} \big)
\le C \s^{2}s\log^{\frac{2}{a}}(ed).
	\end{equation*} 
This is immediate by Theorem~\ref{theorem_adaptiveupperbound} and the fact that
$
| \hat{N}_{\sf exp}^{\circ} - \|\bt\|_2 |^{2} \le  \| \hat{\bt} - \bt\|^2_2 
$
for the plug-in estimator $\hat{N}_{\sf exp}^\circ=\|\hat \bt\|_2$.
}

\subsection{Proof of Proposition~\ref{prop:gao}}
Denote by $G$ the cdf of $(\s\xi_1)^2$ and by $G_d$ the empirical cdf of $((\s\xi_i)^2:  i\not\in S)$, where $S$ is the support of $\bt$. Let $M$ be the median of $G$, that is $G(M)=1/2$. By the definition of $\hat{M}$,
\begin{equation}\label{prop:gao_1}
|F_d(\hat{M})-1/2|\le |F_d(M)-1/2|.
\end{equation}
It is easy to check that $|F_d(x)-G_d(x)|\le s/d$ for all $x>0$. Therefore,
\begin{equation}\label{prop:gao_2}
|G_d(\hat{M})-1/2|\le |G_d(M)-1/2| +2s/d.
\end{equation}
The DKW inequality \cite[page 99]{Wasserman}, yields that
$
\prob(\sup_{x\in \RR}|G_d(x)-G(x)|\ge u)\le 2e^{-2u^2(d-s)}$  for all $u>0$.
Fix $t>0$ such that  $\sqrt{\frac{t}{d}}+ \frac{s}{d} \le 1/8$, and consider the event 
$$\mathcal A:=\left\{\sup_{x\in \RR}|G_d(x)-G(x)|\le \sqrt{\frac{t}{2(d-s)}}\right\}. 
$$ 
Then, $\prob(\mathcal A) \ge 1-2e^{-t}$. On the event $\mathcal A$, we have
\begin{equation}\label{prop:gao_3}
|G(\hat{M})-1/2|\le |G(M)-1/2| +2\left( \sqrt{\frac{t}{2(d-s)}}+ \frac{s}{d}\right) \le 2\left( \sqrt{\frac{t}{d}}+ \frac{s}{d}\right) \le \frac14,
\end{equation}
where the last two inequalities are due to the fact that $G(M)=1/2$ and to the assumption about $t$. Notice that 
\begin{equation}\label{prop:gao_4}
|G(\hat{M})-1/2|= |G(\hat{M})-G(M)| =  \big|F(\hat{M}/\s^{2})-F(M/\s^{2})\big|.
\end{equation}
Using \eqref{prop:gao_3},  \eqref{prop:gao_4} and the fact that $M= \s^2F^{-1}(1/2)$ we obtain that, on the event $\mathcal A$, 
\begin{equation}\label{prop:gao_5}
F^{-1}(1/4) \le \hat{M}/\sigma^{2} \le F^{-1}(3/4).
\end{equation}
This and \eqref{prop:gao_4} imply
\begin{equation}\label{prop:gao_6}
|G(\hat{M})-1/2|\ge  c_{**}\big|\hat{M}/\sigma^{2}- M/\sigma^{2}\big|= c_{**}\beta\,\big|\hat{\s}_{\sf med}^{2}/\sigma^{2}- 1\big|.
\end{equation}
where $c_{**}= \min_{x\in[ F^{-1}(1/4), F^{-1}(3/4)]}F'(x)>0$, and $\beta=F^{-1}(1/2)$. Combining the last inequality with 
 \eqref{prop:gao_3}  we get that, on the event $\mathcal A$,  
$$ 
\,\big|\hat{\s}_{\sf med}^{2}/\sigma^{2}- 1\big|\le  c_{*}\left( \sqrt{\frac{t}{d}}+\frac{s}{d}\right)
$$
where $c_{*}>0$ is a constant depending only on $F$.
Recall that we assumed that $\sqrt{\frac{t}{d}}+ \frac{s}{d} \le 1/8$. 
This condition holds true if $0<t\le t_0=(1/8-\g)^2d$ since $\frac{s}{d} \le \g<1/8$. Thus, 
for all $t>0$ and integers $s,d$ satisfying $0<t\le t_0=(1/8-\g)^2d$ we have 
	\begin{equation}\label{eq:var:prob}
		\sup_{\s>0}\sup_{\|\bt\|_0\le s} \prob_{\bt, F,\s} \left( \Big|\frac{\hat{\s}_{\sf med}^2}{\s^2}-1\Big| \ge c_*\left(\sqrt{\frac{t}{d}}+\frac{s}{d}\right)\right) \le 2e^{-t}.
	\end{equation}
For all $t>t_0$ the left hand side of \eqref{eq:var:prob} does not exceed $2e^{-t_0}=2e^{-(1/8-\g)^2d}$. Combining this remark with  \eqref{eq:var:prob} yields the result of the
proposition in probability. 
We now prove the result in expectation. Set 
$Z=\left|\hat{\s}_{\sf med}^2-\s^{2}\right|/\s^{2}$. We have
$$
\esp_{\bt, F,\s}\left(Z \right) \le c_{*}s/d + \int_{c_{*}s/d}^{c_{*}/8}  \prob_{\bt, F,\s} \left( Z> u \right)du + \esp_{\bt, F,\s}\left(Z\fcar_{Z\ge c_{*}/8 } \right).
$$
Using \eqref{eq:var:prob}, we get
$$
 \int_{c_{*}s/d}^{c_{*}/8} \prob_{\bt, F,\s} \left( Z > u \right)du \leq  \frac{2c_{*}}{\sqrt{d}}\int_{0}^{\infty}e^{-t^{2}}dt \leq \frac{C}{\sqrt{d}}.
$$
As $s < d/2$, one may check that 
$\hat{\s}_{\sf med}^{2+\epsilon} \le \big(\max_{i\not\in S} (\s\xi_i)^2/\beta\big)^{1+\epsilon/2} 
\leq (\s^{2}/\beta)^{1+\epsilon/2} \sum_{i=1}^{d}|\xi_{i}|^{2+\epsilon}  $. 
Since $\esp|\xi_{1}|^{2+\epsilon}<\infty$ this yields
$
\esp_{\bt, F,\s}\left(Z^{1+\epsilon}\right) \le C d. 
$
It follows that 
\begin{align*}
\esp_{\bt, F,\s}\left(Z\fcar_{Z\ge c_{*}/8 } \right) &\le 
\left(\esp_{\bt, F,\s}\left(Z^{1+\epsilon}\right) \right)^{1/(1+\epsilon)}
\prob_{\bt, F,\s}\left(Z\ge c_{*}/8  \right)^{\epsilon/(1+\epsilon)}
\le C d e^{-d/C}. 
\end{align*}
Combining  the last three displays yields the desired bound in expectation.

\subsection{Proof of part (i) of Proposition~\ref{prop:variance:gauss}}
In this proof, we write for brevity $\esp=\esp_{\bt, \s, \calN(0,1)}$ and 
$\prob=\prob_{\bt, \s, \calN(0,1)}$. 
Set
$$\varphi_d(t)=\frac{1}{d}\sum_{i=1}^d e^{itY_j},\quad \varphi(t)= \esp (\varphi_d(t)), 
\quad \varphi_0(t)=e^{-\frac{t^2\s^{2}}{2}}.$$
Since $s/d<1/8$ and 
$\varphi(t)=\varphi_0(t)\big( 1-\frac{|S|}{d}+\frac{1}{d}\sum_{j\in S} \exp(i\t_j t)\big)$, we have 
\begin{equation}\label{propfourier1}
\frac34\varphi_{0}(t)\le\Big(1-\frac{2s}{d}\Big)\varphi_0(t)\le |\varphi(t)|\le \varphi_0(t), \quad \forall t\in \RR. 
\end{equation}
Consider the events 
$$\calB_1=\Big\{ \s^{2}/2 \leq \tilde{\s}^{2}\leq 3\s^{2}/2 \Big\}\quad\text{and}\quad 
\calA_1=\left\{ \sup_{v\in \RR} |\varphi_d(v) -\varphi(v)|\le  \sqrt{\frac{\log(ed)}{d}}\right\}.$$ 
Note also that under the event $\calB_1$
\begin{equation}\label{propfourier1a}
\varphi_{0}(\hat{t}_{1})\ge \frac{1}{2}(es/\sqrt{d}+1)^{-1/4} . 
\end{equation}	
By Proposition \ref{proposition_over}, the event $\calB_1$ is of probability at least $1-e^{-cd}$  
if the tuning parameter $\gamma$ in the definition of $\tilde{\s}^{2}$ is small enough. 
 Using the bounded difference inequality, it is not hard to check that $\calA_1$ holds with probability 
at least $1-C/d$.  Moreover,
\begin{equation}\label{emp}
\esp\Big(\sqrt{d}\sup_{v\in \RR} |\varphi_d(v) -\varphi(v)|\Big)\le C.
\end{equation}
By the definition of $\hat{\s}^{2}$, on the event 
${\cal D} = \{|\varphi_{d}(\hat{t}_{1})|> (es/\sqrt{d}+1)^{-1}/4\}$ we have
$\hat{\s}^{2}=\tilde v^2 \le 8 \tilde{\s}^{2}$.
First, we bound the risk restricted to ${\cal D}\cap \calB_1^{c}$. We have
 $$
 \mathbf{E}\big(|\hat{\s}^{2}-\s^{2}|\fcar_{{\cal D}\cap\calB_{1}^{c}}\big)
\leq \mathbf{E}\big(|8\tilde{\s}^{2}+\s^{2}|\fcar_{\calB_{1}^{c}}\big).
 $$
Thus, using the Cauchy-Schwarz inequality and Proposition \ref{proposition_over} we find  
 \begin{equation}\label{eq:B_comp}
 \mathbf{E}\big(|\hat{\s}^{2}-\s^{2}|\fcar_{{\cal D}\cap\calB_{1}^{c}}\big)
\leq C\s^{2}e^{-d/C}\leq \frac{C'\s^{2}}{\sqrt{d}}.
 \end{equation}
Next, we bound the risk restricted to ${\cal D}^{c}$. It will be useful to note 
that  $ \calA_{1} \cap \calB_{1} \subset {\cal D}$. Indeed, on 
$ \calA_{1} \cap \calB_{1}$, using \eqref{propfourier1a} and the assumption $s<d/8$ we have 
	$$
|\varphi_{d}(\hat{t}_{1})| \geq \frac34\varphi_{0}(\hat{t}_{1}) - \sqrt{\frac{\log{(ed)}}{d}} 
\geq  \frac{3}{8({es}/{\sqrt{d}}+1)^{1/4}} - \sqrt{\frac{\log{(ed)}}{d}} >
\frac{1}{4({es}/{\sqrt{d}}+1)} . 
$$
Thus, applying again the Cauchy-Schwarz inequality and Proposition \ref{proposition_over} we find 
\begin{align}\label{eq:sigma_nul}
\mathbf{E}\big(|\hat{\s}^{2}-\s^{2}|\fcar_{{\cal D}^{c}}\big)
&=\mathbf{E}\big(|\tilde{\s}^{2}-\s^{2}|\fcar_{{\cal D}^{c}}\big)
\le \left(\mathbf{E}\big(|\tilde{\s}^{2}-\s^{2}|^2\big)\right)^{1/2}\left(\prob ({\cal D}^{c})\right)^{1/2}
\\ \nonumber
&\leq C\s^{2}\sqrt{ \prob (\calA_{1}^{c})+\prob (\calB_{1}^{c})  } \le 
C\s^{2}\sqrt{ \frac{1}{d} + e^{-cd} }\leq 
\frac{C'\s^{2}}{\sqrt{d}}
.
\end{align}
To complete the proof, it remains to handle the risk restricted to the event 
$\mathcal{C} = {\cal D}\cap\calB_{1}$. On this event, $\hat{\s}^{2}=\tilde v^2 = -2\log (|\varphi_d(\hat{t}_1)|)/\hat{t}_1^2$, so that
\begin{equation}\label{propfourier3}
|\hat{\s}^2-\s^2|\le \Big|\frac{2\log (|\varphi_d(\hat{t}_1)|)}{\hat{t}_1^2}-
\frac{2 \log (|\varphi(\hat{t}_1)|)}{\hat{t}_1^2}\Big|+ \Big|\frac{2 \log (|\varphi(\hat{t}_1)|)}{\hat{t}_1^2}+\s^2\Big|.
\end{equation} 
We will use the  inequality 
\begin{equation*}
\big|\log (|\varphi_j(t)|)-\log (|\varphi(t)|)\big|\le  \frac{|\varphi_j(t)-\varphi(t)|}
{|\varphi(t)| \wedge |\varphi_{j}(t)| }\,,\quad \forall t\in \RR, \quad j\in\{0,d\}.
\end{equation*}
Since $\s^2 =-{2 \log (\varphi_0(\hat{t}_1))}/{\hat{t}_1^2}$, it follows from the previous inequality and from
\eqref{propfourier1}  that
  \begin{equation*}
\Big|\frac{2 \log (|\varphi(\hat{t}_1)|)}{\hat{t}_1^2}+\s^2\Big|\le 
\frac{4s \varphi_{0}(\hat{t}_1)}
{d\,\hat{t}_1^2 (|\varphi(\hat{t}_1)| \wedge \varphi_{0}(\hat{t}_1)) }
\le  \frac{Cs}{d\,\hat{t}_1^2} \le 
\frac{C's\tilde{\s}^{2}}{d\log(16(es/\sqrt{d}+1))}.
\end{equation*} 
Therefore,
\begin{equation}\label{eq:part1}
\mathbf{E}\Big(\Big|\frac{2 \log (|\varphi(\hat{t_1})|)}{\hat{t}_1^2}+\s^2\Big|
\fcar_{\mathcal{C}}\Big) \leq \frac{Cs\s^{2}}{d\log(es/\sqrt{d}+1)}.
\end{equation}
Next, recall that on $\mathcal{D}$ we have $|\varphi_{d}(\hat{t}_{1})|> (es/\sqrt{d}+1)^{-1}/4$. Using this fact and the inequalities \eqref{propfourier1}, \eqref{propfourier1a} 
we obtain
\begin{align*}
\Big|\frac{\log (|\varphi_d(\hat{t}_1)|)}{\hat{t}_1^2}-\frac{ \log (|\varphi(\hat{t_1})|)}{\hat{t}_1^2}\Big|
\fcar_{\mathcal{C}}
     &\le\frac{\sup_{v\in \RR} |\varphi_d(v) -\varphi(v)|}
{\hat{t}^{2}_{1}(|\varphi(\hat{t}_{1})| \wedge |\varphi_{d}(\hat{t}_{1})|) }\fcar_{\mathcal{C}} 
\\
&\leq 
\frac{C\s^{2}U}{\sqrt{d}\,\log(es/\sqrt{d}+1)}\left(\frac{es}{\sqrt{d}}+1\right)\, , 
\end{align*}
where $U=\sqrt{d}\,\sup_{v\in \RR} |\varphi_d(v) -\varphi(v)|$. 
Bounding $\esp(U)$ by \eqref{emp} we finally get
  \begin{equation}\label{eq:part3}
  \esp\left[\Big|\frac{\log (|\varphi_d(\hat{t}_1)|)}{\hat{t}_1^2}-
\frac{ \log (|\varphi(\hat{t_1})|)}{\hat{t}_1^2}\Big|\fcar_{\mathcal{C}}\right]\le 
C\s^2\max\left(\frac1{\sqrt{d}},\frac{s}{d\log (es/\sqrt{d}+1)}\right) .  
\end{equation} 
  We conclude by combining inequalities
   \eqref{eq:B_comp} - \eqref{eq:part3}.

\subsection{Proof of Theorems~\ref{theorem_upperbound_noise_subgaussian} and~\ref{theorem_upperbound_noise_polynomial}}

Let $\|\bt\|_0\le s$ and denote by $S$ the support of $\bt$. Note first that, by the definition of $\hat{\s}^2$, 
\begin{equation}\label{upper_crucial}
	\frac{\s^2}{d}\sum_{i=1}^{d-2s} \xi_{(i)}^2 \le \hat{\s}^2\le \frac{\s^2}{d}\sum_{i\in S^c} \xi_{i}^2,
\end{equation}
where $\xi_{(1)}^2\le \cdots\le \xi_{(d)}^2$ are the ordered values of $\xi_1^2,\dots,\xi_d^2$. Indeed, the right hand inequality in \eqref{upper_crucial} follows from the relations
$$
\sum_{k=1}^{d-s} Y_{(k)}^2 = \min_{J: |J|=d-s} \sum_{i\in J} Y_{(i)}^2  \le \sum_{i\in S^c}Y_{(i)}^2 = \sum_{i\in S^c} \s^2\xi_{i}^2.
$$
To show the left hand inequality in \eqref{upper_crucial}, notice that at least $d-2s$ among the $d-s$ order statistics $Y_{(1)}^2, \dots,Y_{(d-s)}^2$ correspond to observations $Y_k$ of pure noise, \ie $Y_k=\s \xi_k$.  The sum of squares of such observations is bounded from below by the sum of the smallest $d-2s$ values $\s^2\xi_{(1)}^2, \dots, \s^2\xi_{(d-2s)}^2$ among $\s^2\xi_{1}^2, \dots, \s^2\xi_{d}^2$.

Using \eqref{upper_crucial} we get
\begin{equation*}
	\Big(\hat{\s}^2-\frac{\s^2}{d}\sum_{i=1}^d \xi_i^2 \Big)^2 \le \frac{\s^4}{d^2} \Big( \sum_{i=d-2s+1}^d \xi_{(i)}^2 \Big)^2,
\end{equation*}
so that
\begin{equation*}
	\esp_{\bt,P_\xi,\s} \Big(\hat{\s}^2-\frac{\s^2}{d}\sum_{i=1}^d \xi_i^2 \Big)^2 \le \frac{\s^4}{d^2} \Big(\sum_{i= 1}^{2s} \sqrt{ \esp \xi_{(d-i+1)}^4} \Big)^2.
\end{equation*}
Then
\begin{eqnarray*}
\esp_{\bt,P_\xi,\s} (\hat{\s}^2-\s^2)^2 &\le& 2\esp_{\bt,P_\xi,\s} \Big(\hat{\s}^2-\frac{\s^2}{d}\sum_{i=1}^d \xi_i^2 \Big)^2+ 2\esp_{\bt,P_\xi,\s} \Big(\frac{\s^2}{d}\sum_{i=1}^d \xi_i^2 -\s^2\Big)^2
\\
&\le &\frac{2\s^4}{d^2} \Big(\sum_{i= 1}^{2s} \sqrt{ \esp \xi_{(d-i+1)}^4} \Big)^2 + \frac{2\s^4 \esp(\xi_1^4)}{d}.
\end{eqnarray*}
Note that under assumption \eqref{definition_subgaussian} we have $\esp(\xi_1^4)<\infty$ and Lemmas~\ref{lemma_esp4_subgaussian} and \ref{lemma:sum} yield
\begin{align*}
\sum_{i= 1}^{2s} \sqrt{ \esp \xi_{(d-i+1)}^4} \le \sqrt{C}	 \sum_{i=1}^{2s} \log^{2/a}\big(ed/i\big) \le C'\sqrt{C} s \log^{2/a}\Big(\frac{ed}{2s}\Big).
\end{align*}
This proves Theorem~\ref{theorem_upperbound_noise_subgaussian}. To prove Theorem~\ref{theorem_upperbound_noise_polynomial}, we act analogously by using Lemma~\ref{lemma_esp4_polynomial} and the fact that  $\esp(\xi_1^4)<\infty$
under assumption \eqref{definition_polynomial} with $a>4$.


{
\subsection{Proof of Theorem~\ref{theorem_adaptiveupperbound_variance}}
With the same notation as in the proof of Theorem~\ref{theorem_adaptiveupperbound}, we have
\begin{equation}\label{eq54}
	\hat{\s}^2-\s^2 = \frac{\s^2}d \big( \|\bxi\|_2^2-d \big) + \frac1d 
\left(\|\bu\|_2^2 - 2\s \bu^T \bxi\right). 
\end{equation}
It follows from (\ref{combination}) that
$$
\|\bu\|_2^2 + 2\s |\bu^T \bxi | \le 3 \s  |\bu^T \bxi | +\frac{\tilde{\s}}{2}\Big\{ \Big(\sum_{j=1}^s \l_j^2\Big)^{1/2} \| \bu\|_2 -\sum_{j=s+1}^d \l_j |u|_{(d-j+1)} \Big\}.
$$ 
Arguing as in the proof of Theorem ~\ref{theorem_adaptiveupperbound}, we obtain
$$
\|\bu\|_2^2 + 2\s |\bu^T \bxi | \le \Big( U_1
+ \frac{\tilde{\s}}{2}\Big(\sum_{j=1}^s \l_j^2\Big)^{1/2}  
+ U_2\Big) \| \bu\|_2,
$$
where
$$
U_1=3 \s \Big(\sum_{j=1}^s |\xi|_{(d-j+1)} ^2\Big)^{1/2}, \quad U_2= \Big(\sum_{j=s+1}^d \left(3\s |\xi|_{(d-j+1)} 
- \frac{\tilde{\s}}{2}\l_{j} \right)^{2}_{+} \Big)^{1/2}
$$
Using the Cauchy-Schwarz inequality, Proposition \ref{proposition_over} and \eqref{eqq1} and writing for brevity 
$\esp=\esp_{\bt,P_\xi,\s}$ we find
\begin{equation*}
		\esp\Big(\tilde{\s}
\Big(\sum_{j=1}^s \l_j^2\Big)^{1/2} \|\bu\|_2\Big)\le 
\Big(\sum_{j=1}^s \l_j^2\Big)^{1/2} \sqrt{\esp(\tilde{\s}^2)} \sqrt{\esp(\|\bu\|^{2}_2)}
 \le C \s^{2}\sum_{j=1}^s \l_j^2.
		\end{equation*}
Since $\mathbf{E}(\xi_{1}^{4})<\infty$ we also have $\esp\big| \|\bxi\|_2^2-d \big|\le C\sqrt{d}$.
Finally, using again \eqref{eqq1} we get, for $k=1,2$, 
$$
\esp(U_k \|\bu\|_{2})\le  \sqrt{\esp(\|\bu\|^{2}_2)} \sqrt{\esp(U_k^2)}
\le \s \Big(\sum_{j=1}^s \l_j^2\Big)^{1/2} \sqrt{\esp(U_k^2)}\le C \s^{2}\sum_{j=1}^s \l_j^2,
$$
where the last inequality follows from the same argument 
as in the proof of Theorem ~\ref{theorem_adaptiveupperbound}.
These remarks together with \eqref{eq54} imply 
$$
\esp\left(|\hat{\s}^2-\s^2 |\right) \leq \frac{C}{d}\Big(\s^{2}\sqrt{d} + 
\s^{2}\sum_{j=1}^s \l_{j}^{2}\Big). 
$$
We conclude the proof by bounding $\sum_{j=1}^s \l_{j}^{2}$ in the same way 
as in the end of the proof of Theorem~\ref{theorem_adaptiveupperbound}.
}
\section{Proofs of the lower bounds}

\subsection{Proof of Theorems~\ref{theorem_lowerbound_noise_subgaussian} and~\ref{theorem_lowerbound_noise_polynomial} and part (ii) of Proposition ~\ref{prop:variance:gauss}}\label{subsection_proof_lowerbound_noise}
Since we have $\ell(t)\ge \ell(A)\fcar_{t\ge A}$ for any $A>0$, it is enough to prove the theorems for the indicator loss $\ell(t)=\fcar_{t\ge 1}$. This remark is valid for all the proofs of this section and will not be further repeated. 

(i) We first prove the lower bounds with the rate $1/{\sqrt{d}}$ in Theorems~\ref{theorem_lowerbound_noise_subgaussian} and~\ref{theorem_lowerbound_noise_polynomial}. Let $f_0:\RR\to [0, \infty)$ be a probability density with the following properties: $f_0$ is continuously differentiable, symmetric about 0, supported on $[-3/2,3/2]$, with variance 1 and finite Fisher information $I_{f_0}= \int (f_0'(x))^2(f_0(x))^{-1}dx$. The existence of such $f_0$ is shown in Lemma~\ref{lemma_density}. Denote by $F_0$ the probability distribution corresponding to $f_0$.
Since $F_0$ is zero-mean, with variance 1 and supported on $[-3/2,3/2]$ it belongs to $\mathcal{G}_{a,\tau}$ for any $\tau>0$, $a>0$, and to $\mathcal{P}_{a,\tau}$ for any $\tau>0$, $a\ge 2$.  
Define $\prob_0=\prob_{0,F_0,1}$, $\prob_1=\prob_{0,F_0,\s_1}$ where $\s_1^2=1+c_0/\sqrt{d}$ and $c_0>0$ is a small constant to be fixed later.  
 Denote by $H(\prob_1,\prob_0)$ the Hellinger distance between $\prob_1$ and $\prob_0$. We have
\begin{equation}\label{hellgr}
	H^2(\prob_1,\prob_0) = 2\big(1-(1-h^2/2)^d\big)
\end{equation}
where $h^2=\int (\sqrt{f_0(x)}-\sqrt{f_0(x/\s_1)/\s_1})^2 dx$. By Theorem 7.6. in~\cite{Ibragimov},
$$
h^2 \le \frac{(1-\s_1)^2}{4}\sup_{t\in [1,\s_1]} I(t)
$$
where $I(t)$ is the Fisher information corresponding to the density $f_0(x/t)/t$, that is $I(t)= t^{-2}I_{f_0}$.
It follows that $h^2\le {\bar c}c_0^2/d$ where ${\bar c}>0$ is a constant. This and \eqref{hellgr} imply that for $c_0$ 
small enough we have $H(\prob_1,\prob_0)\le 1/2$.
Finally, choosing such a small $c_0$ and using Theorem~2.2(ii) in~\cite{Tsybakov2009} we obtain
 \begin{eqnarray*}
 	&&\inf_{\hat{T}} \max\Big\{ \prob_0 \Big(\Big|\hat{T}-1\Big|\ge \frac{c_0}{2(1+c_0)\sqrt{d}}\Big),  \prob_1 \Big(\Big|\frac{\hat{T}}{\s_1^2}-1\Big|\ge\frac{c_0}{2(1+c_0)\sqrt{d}}\Big)\Big\}\\
	&& \ge  
	\inf_{\hat{T}} \max\Big\{ \prob_0 \Big(|\hat{T}-1|\ge \frac{c_0}{2\sqrt{d}}\Big),  \prob_1 \Big(|\hat{T}-\s_1^2|\ge \frac{c_0}{2\sqrt{d}}\Big)\Big\}
	 \ge \frac{1-H(\prob_1,\prob_0)}{2}\ge \frac14.
 \end{eqnarray*}
  
(ii) We now prove the lower bound with the rate $\frac{s}{d}\log^{2/a}(ed/s)$ in Theorem~\ref{theorem_lowerbound_noise_subgaussian}. It is enough to conduct the proof for $s\ge s_0$ where $s_0>0$ is an 
arbitrary absolute constant. Indeed, for $s\le s_0$ we have $\frac{s}{d}\log^{2/a}(ed/s) \le C/\sqrt{d}$ where $C>0$ is  an absolute constant and thus Theorem~\ref{theorem_lowerbound_noise_subgaussian} follows already from
the lower bound with the rate $1/\sqrt{d}$ proved in item (i). Therefore, in the rest of this proof we assume without loss of generality that  $s\ge 32$.
 
We take $P_\xi= U$ where $U$ is the Rademacher  distribution, that is the uniform distribution on $\{-1,1\}$.  Clearly, $U\in\mathcal{G}_{a,\tau}$. Let $\d_1,\ldots,\d_d$ be i.i.d. Bernoulli random variables with probability of success $\prob (\d_1=1)=\frac{s}{2d}$, and let 
$
 	\e_1,\ldots,\e_d
 $ 
be i.i.d. Rademacher random variables
 that are independent of $(\d_1,\ldots,\d_d)$.
Denote by $\mu$ the distribution of $(\a\d_1\e_1,\ldots,\a\d_d\e_d)$ where $\a=(\tau/2)\log^{1/a}(ed/s)$. Note that $\mu$ is not necessarily supported on $\T_s=\{\bt\in\RR^d\suchthat\|\bt\|_0\le s\}$ as the number of nonzero components of a vector drawn from $\mu$ can be larger than $s$. Therefore, we consider a restricted to $\T_s$ version of $\mu$ defined by  
 \begin{equation}\label{definition_barmu}
 	\bar{\mu}(A) = \frac{\mu(A\cap\T_s)}{\mu(\T_s)}
 \end{equation}
 for all Borel subsets $A$ of $\RR^d$.
Finally, we introduce two mixture probability measures 
\begin{equation}\label{definition_apriori}
	{\mathbb P}_\mu = \int \prob_{\bt,U,1} \, \mu(d\bt) \quad\text{and}\quad {\mathbb P}_{\bar{\mu}} = \int \prob_{\bt,U,1} \, \bar\mu(d\bt).
\end{equation}
Notice that there exists a probability measure $\tilde P\in \mathcal{G}_{a,\tau}$ such 
that
\begin{equation}\label{crucial}
	{\mathbb P}_\mu = \prob_{0,\tilde P,\s_0} 
\end{equation}
where $\s_0>0$ is defined by
\begin{equation}\label{sigma0}
\s_0^2=1+\frac{\tau^2 s}{8 d}\log^{2/a}(ed/s) \le 1+\frac{\tau^2}{8}.
\end{equation}
Indeed, $\s_0^2=1+\frac{\a^2s}{2d}$ is the variance of zero-mean random variable $\a\d\e+\xi$, where $\xi\sim U$, $\e\sim U$, $\d\sim \mathcal{B}\big(\frac{s}{2d}\big)$ and $\e,\xi,\d$ are jointly independent. Thus, to prove \eqref{crucial} it is enough to show that, for all $t\ge2$,
\begin{equation}\label{probb}
	\prob\big((\tau/2)\log^{1/a}(ed/s) \,\d\e + \xi>t \s_0\big) \le e^{-(t/\tau)^a}.
\end{equation}
But this inequality immediately follows from the fact that  for $t\ge2$
 the probability in \eqref{probb} is smaller than
\begin{align}
	\prob(\e=1, \d=1)\,\fcar_{(\tau/2)\log^{1/a}(ed/s)>t-1} \le \frac{s}{4d}\fcar_{\tau\log^{1/a}(ed/s)>t} \le e^{-(t/\tau)^a}.
\end{align}
Now, for any estimator $\hat T$ and any $u>0$ we have
\begin{eqnarray}\nonumber
		 &&\sup_{P_\xi\in\mathcal{G}_{a,\tau}} \sup_{\s>0} \sup_{\|\bt\|_0\le s}\prob_{\bt,P_\xi,\s} \Big( \Big| \frac{ \hat{T}}{\s^2} -1\Big| \ge u \Big) \\ \nonumber
		 &&\qquad \ge \max\Big\{ \prob_{0,\tilde P,\s_0} ( | \hat{T} -\s_0^2| \ge \s_0^2u),  \int \prob_{\bt,U,1} ( | \hat{T} -1| \ge u)  {\bar \mu}(d\bt)\Big \}\\
		 &&\qquad \ge \max\Big\{ {\mathbb P}_\mu( | \hat{T} -\s_0^2| \ge \s_0^2u),  {\mathbb P}_{\bar \mu} ( | \hat{T} -1| \ge \s_0^2 u) \Big \} \label{lowerr}
	\end{eqnarray}
where the last inequality uses \eqref{crucial}. Write $\s_0^2=1+2\phi$ where 
$\phi= \frac{\tau^2 s}{16 d}\log^{2/a}(ed/s)$ and choose $u=\phi/\s_0^2 \ge \phi/(1+\tau^2/8)$. Then, the expression in \eqref{lowerr} is bounded from below by the probability of error in the problem of distinguishing between two simple hypotheses ${\mathbb P}_{\mu}$ and ${\mathbb P}_{\bar \mu}$, for which Theorem~2.2 in~\cite{Tsybakov2009} yields
\begin{eqnarray}
		 \max\Big\{ {\mathbb P}_\mu( | \hat{T} -\s_0^2| \ge \phi),  {\mathbb P}_{\bar \mu} ( | \hat{T} -1| \ge \phi) \Big \} \ge \frac{1-V({\mathbb P}_{ \mu},{\mathbb P}_{\bar\mu})}{2}
		 \label{lowerr1}
	\end{eqnarray}
where $V({\mathbb P}_{ \mu},{\mathbb P}_{\bar\mu})$ is the total variation distance between ${\mathbb P}_{\mu}$ and ${\mathbb P}_{\bar \mu}$. The desired lower bound follows from \eqref{lowerr1} and Lemma~\ref{lemma_TV} for any $s\ge 32$.

(iii) Finally, we prove the lower bound with the rate $\tau^2(s/d)^{1-2/a}$ in Theorem~\ref{theorem_lowerbound_noise_polynomial}. Again, we do not consider the case $s\le 32$ since in this case the rate 
$1/\sqrt{d}$ is dominating and Theorem~\ref{theorem_lowerbound_noise_polynomial} follows from item (i) above. For $s\ge 32$, the proof uses the same argument as in item (ii) above but we choose $\a=(\tau/2)(d/s)^{1/a}$.
Then the variance of  $\a\d\e+\xi$ is equal to
$$\s_0^2=1+ \frac{\tau^2(s/d)^{1-2/a}}{8}.
$$
Furthermore, with this definition of $\s_0^2$ there exists $\tilde P\in {\cal P}_{a,\tau}$ such that \eqref{crucial} holds. Indeed, 
analogously to \eqref{probb} we now have, for all $t\ge 2$, 
\begin{align}
	\prob\big(\a \,\d\e + \xi>t \s_0\big) &\le \prob(\e=1, \d=1)\,\fcar_{(\tau/2)(d/s)^{1/a}>t-1} \le \frac{s}{4d}\fcar_{\tau(d/s)^{1/a}>t} \le (t/\tau)^a.
\end{align}
To finish the proof, it remains to repeat the argument of \eqref{lowerr} and \eqref{lowerr1} with $\phi=\frac{\tau^2(s/d)^{1-2/a}}{16}.$
\subsection{Proof of Theorem~\ref{theorem_lowerbound_norm_subgaussian}}

We argue similarly to the proof of Theorems~\ref{theorem_lowerbound_noise_subgaussian} and~\ref{theorem_lowerbound_noise_polynomial}, in particular, we
set $\a=(\tau/2)\log^{1/a}(ed/s)$ when proving the bound on the class ${\cal G}_{a,\tau}$, and $\a=(\tau/2)(d/s)^{1/a}$ when proving the bound on~${\cal P}_{a,\tau}$. 
In what follows, we only deal with the class ${\cal G}_{a,\tau}$ since the proof  for~${\cal P}_{a,\tau}$ is analogous. Consider the measures $\mu$, $\bar{\mu}$, ${\mathbb P}_{\mu}$, ${\mathbb P}_{\bar{\mu}}$ and 
$\tilde{P}$ defined in Section~\ref{subsection_proof_lowerbound_noise}. Similarly to \eqref{lowerr}, for any estimator $\hat T$ and any $u>0$ we have
\begin{eqnarray}\nonumber
&&\sup_{P_\xi\in\mathcal{G}_{a,\tau}} \sup_{\s>0} \sup_{\|\bt\|_0\le s}\prob_{\bt,P_\xi,\s} \big( | \hat{T} -\|\bt\|_2| \ge \s u \big) \\ \nonumber
&&\qquad \ge \max\Big\{ \prob_{0,\tilde P,\s_0} ( | \hat{T}| \ge \s_0 u),  \int \prob_{\bt,U,1} ( | \hat{T} -\|\bt\|_2| \ge u)  {\bar \mu}(d\bt)\Big \}\\
&&\qquad \ge \max\Big\{ {\mathbb P}_\mu( | \hat{T}| \ge \s_0u),  {\mathbb P}_{\bar \mu} ( | \hat{T} -\|\bt\|_2| \ge \s_0u) \Big \} \nonumber
\\
&&\qquad \ge  \max\Big\{ {\mathbb P}_\mu( | \hat{T}| \ge \s_0u),  {\mathbb P}_{\bar \mu} ( | \hat{T}| < \s_0u, \|\bt\|_2\ge 2\s_0u) \Big \}\nonumber
\\
&&\qquad \ge  \min_{B}\,\max\big\{ {\mathbb P}_\mu( B),  {\mathbb P}_{\bar \mu} ( B^c) - {\bar \mu}( \|\bt\|_2< 2\s_0u)\big \}
\nonumber
\\
&&\qquad \ge  \min_{B}\,\frac{ {\mathbb P}_\mu( B) +  {\mathbb P}_{\bar \mu} ( B^c)}{2} - \frac{{\bar \mu}( \|\bt\|_2< 2\s_0u)}2 \phantom{\Big\}}
\label{lowerr2}
\end{eqnarray}
where $\s_0$ is defined in \eqref{sigma0},  $U$ denotes the Rademacher law and $\min_{B}$ is the minimum over all Borel sets.
The third line in the last display is due to \eqref{crucial} and to the inequality $\s_0\ge1$.  Since 
$\min_{B}\,\big\{ {\mathbb P}_\mu ( B) +  {\mathbb P}_{\bar \mu} ( B^c)\big\} = 1-V({\mathbb P}_{ \mu},{\mathbb P}_{\bar\mu})$, we get
\begin{eqnarray}\label{lowerr3aa}
	&&\sup_{P_\xi\in\mathcal{G}_{a,\tau}} \sup_{\s>0} \sup_{\|\bt\|_0\le s}\prob_{\bt,P_\xi,\s} \big( | \hat{T} -\|\bt\|_2|/\s \ge u\big) 
	\ge \frac{1-V({\mathbb P}_{ \mu},{\mathbb P}_{\bar\mu}) - \bar \mu (\|\bt\|_2 < 2\s_0u)}{2}.
\end{eqnarray}
Consider first the case $s\ge 32$. Set $u=\frac{\a\sqrt{s}}{4\s_0}$. Then \eqref{eq1:lemma_TV} and \eqref{eq1:lemma_concentration_barmu} imply that  
$$
V({\mathbb P}_{ \mu},{\mathbb P}_{\bar\mu}) \le  e^{-\frac{3s}{16}}, \quad  \bar \mu (\|\bt\|_2 < 2\s_0u)\le 2e^{-\frac{s}{16}},
$$
which, together with \eqref{lowerr3aa} and the fact that $s\ge 32$ yields the result. 
 
Let now $s< 32$. Then we set $u=\frac{\a\sqrt{s}}{8\sqrt{2}\s_0}$.  It follows from \eqref{eq2:lemma_TV} and \eqref{eq2:lemma_concentration_barmu} that 
\begin{eqnarray*}\label{lowerr3a}
1-V({\mathbb P}_{ \mu},{\mathbb P}_{\bar\mu}) - \bar \mu (\|\bt\|_2 < 2\s_0u)\ge \prob\Big(\mathcal{B}\big(d,\frac{s}{2d}\big)= 1\Big) = \frac{s}{2}\Big(1-\frac{s}{2d}\Big)^{d-1}.
\end{eqnarray*}
It is not hard to check that the minimum of the last expression over  all integers $s,d$ such that $1\le s < 32$, $s\le d$, is bounded from below by a positive number independent of $d$. We conclude by combining these remarks with \eqref{lowerr3aa}.


%
%

{
\subsection{Proof of part (ii) of Proposition~\ref{prop:norm:known_sigma} and part (ii) of Proposition~\ref{prop:norm:poly:known_sigma}}
We argue similarly to the proof of Theorems~\ref{theorem_lowerbound_noise_subgaussian} and~\ref{theorem_lowerbound_noise_polynomial}, in particular, we
set $\a=(\tau/2)\log^{1/a}(ed/s)$ when proving the bound on the class ${\cal G}_{a,\tau}$, and $\a=(\tau/2)(d/s)^{1/a}$ when proving the bound on~${\cal P}_{a,\tau}$. 
In what follows, we only deal with the class ${\cal G}_{a,\tau}$ since the proof  for~${\cal P}_{a,\tau}$ is analogous. 
Without loss of generality we assume that $\s=1$.

To prove the lower bound with the rate $\phi^{\circ}_{\sf exp}(s,d)$, we only need to prove 
it for $s$ such that $s \le c_{0}\sqrt{d}/\log^{2/a}(ed) $ 
with any small absolute constant $c_{0}>0$, since the rate is increasing with $s$. 

Consider the measures $\mu$,  $\bar{\mu}$, ${\mathbb P}_{\mu}$, ${\mathbb P}_{\bar{\mu}}$ 
defined in Section~\ref{subsection_proof_lowerbound_noise} with $\sigma_0=1$.  
Let $\xi_1$ be distributed with c.d.f. $F_0$ defined in  item (i) of the proof of 
Theorems~\ref{theorem_lowerbound_noise_subgaussian} and~\ref{theorem_lowerbound_noise_polynomial}. 
Using the notation as in the proof of Theorems~\ref{theorem_lowerbound_noise_subgaussian}
and~\ref{theorem_lowerbound_noise_polynomial},  we define $\tilde{P}$ 
as the distribution of  $\tilde{\xi}_1=\s_1\xi_1+\a\delta_1 \e_1$ with    
$\s_1^2=(1+\a^2s/(2d))^{-1}$ where now $\delta_1$ is the Bernoulli random variable with 
$\prob(\delta_1=1)=\frac{s}{2d}(1+\a^2s/(2d))^{-1}$.  
By construction,  $\esp \tilde{\xi}_1=0$ and $\esp \tilde{\xi}_1^2=1$. 
Since the support of $F_0$ is in $[-{3}/{2}, {3}/{2}]$ one can check as in item (ii)
of the proof of Theorems~\ref{theorem_lowerbound_noise_subgaussian} 
and~\ref{theorem_lowerbound_noise_polynomial} 
that $\tilde{P}\in \mathcal{G}_{a,\tau}$.  
Next, analogously to \eqref{lowerr2} - \eqref{lowerr3aa} we obtain that, for any $u>0$,
  $$\sup_{P_\xi \in \mathcal{G}_{a,\tau}}\sup_{\|\t\|_0\le s}\prob_{\bt,P_\xi,1} 
\big( | \hat{T} -\|\bt\|_2|\ge u\big) 
	\ge \frac{1-V({\mathbb P}_{ \bar \mu},P_{0,\tilde{P},1}) - \bar \mu (\|\bt\|_2 
< 2u)}{2}.$$ 
 Let $\mathbf{P}_0$ and $\mathbf{P}_1$ denote the distributions of 
 $(\xi_1,\ldots,\xi_d)$ and of
$(\s_1\xi_1,\ldots,\s_1\xi_d)$, respectively.  Acting as in item (i) of the proof of 
Theorems~\ref{theorem_lowerbound_noise_subgaussian} and~\ref{theorem_lowerbound_noise_polynomial}
and using the bound 
$$|1-\s_1|\le {\alpha^{2}s}/{d}
= \frac{\tau^2}{4} \frac{s}{d}\log^{2/a}(ed/s)\le  C c_0 /\sqrt{d} 
$$ 
we find that $V(\mathbf{P}_0,\mathbf{P}_1)\le H(\mathbf{P}_0,\mathbf{P}_1) \le 2\kappa c_{0}^{2}$ 
for some $\kappa >0$. 
 Therefore, $V(\mathbb{P}_{\mu},P_{0,\tilde{P},1})=
V(\mathbf{P}_0*\mathbf{Q}, \mathbf{P}_1*\mathbf{Q})\le V(\mathbf{P}_0,\mathbf{P}_1)\le 2\kappa c_{0}^{2}$ where $\mathbf{Q}$ denotes the distribution of $(\a\delta_1 \e_1, \ldots,  \a\delta_d \e_d)$. This bound and 
the fact that $V(\mathbb{P}_{\bar{\mu}}, P_{0,\tilde{P},1})\le 
V(\mathbb{P}_{\bar{\mu}},\mathbb{P}_{\mu} )+ V( 
\mathbb{P}_{\mu},P_{0,\tilde{P},1})$ imply
 $$
 \sup_{P_\xi \in \mathcal{G}_{a,\tau}}\sup_{\|\t\|_0\le s}\prob_{\bt,P_\xi,1} \big( | \hat{T} -\|\bt\|_2|\ge u\big) 
 \ge \frac{1-V({\mathbb P}_{ \mu},{\mathbb P}_{\bar\mu}) - \bar \mu (\|\bt\|_2 < 2u)}{2} - \kappa c_{0}^{2}.
 $$
We conclude by repeating the argument after  \eqref{lowerr3aa} in the proof of 
Theorem~\ref{theorem_lowerbound_norm_subgaussian} and choosing $c_{0}>0$ small enough 
to guarantee that the right hand side of the last display is positive. 
}
\subsection{Proof of part (ii) of Proposition~\ref{prop:variance:gauss}}
{
The lower bound with the rate $1/{\sqrt{d}}$ follows from the argument as in item (i) of the proof of Theorems~\ref{theorem_lowerbound_noise_subgaussian} and~\ref{theorem_lowerbound_noise_polynomial} if we replace there $F_{0}$ 
by the standard Gaussian distribution. The lower bound with the rate
 $\frac{s}{d(1+\log_{+}(s^{2}/d))}$ follows from Lemma~\ref{lemma:lowerbound:norm:variance} and 
the lower bound for
estimation of $\|\bt\|_{2}$ in Proposition~\ref{prop:norm:gauss}.
}

\subsection{Proof of Proposition~\ref{proposition_suboptimality}}

Assume that $\bt=0$, $\s=1$ and set 
\begin{equation*}
	\xi_i = \sqrt3 \e_i u_i, 
\end{equation*}
where the $\e_i$'s and the $u_i$ are independent, with Rademacher and uniform distribution on $[0,1]$ respectively. Then note that
\begin{align}\label{43}
	\esp_{0,P_\xi,1} \big(\hat{\s}_*^2-1\big)^2 &\ge \big(\esp_{0,P_\xi,1} (\hat{\s}_*^2)-1\big)^2 = \Big(\esp_{0,P_\xi,1} \Big\{ \hat{\s}_*^2-\frac3d \sum_{i=1}^d u_i^2\Big\} \Big)^2,
\end{align}
since $\esp(u_i^2)=1/3$. Note also that $\hat{\s}_*^2=\frac{3}{d/2}\sum_{i=1}^{d/2} u_{(i)}^2$.
Now,
\begin{align*}
\frac{1}{d/2}\sum_{i=1}^{d/2} u_{(i)}^2-\frac{1}{d}\sum_{i=1}^d u_i^2&=\frac{1}{d}\sum_{i=1}^{d/2}  u_{(i)}^2-\frac{1}{d}\sum_{i=d/2+1}^d u_{(i)}^2\\
&\le \frac{1}{d}\sum_{i=1}^{d/4}  u_{(i)}^2-\frac{1}{d}\sum_{i=3d/4+1}^d u_{(i)}^2\\
&\le \frac{1}{4}(u^2_{(d/4)}-u^2_{(3d/4)}).
\end{align*}
Since $u_{(i)}$ follows a Beta distribution with parameters $(i,d-i+1)$ we have $\esp(u_{(i)}^2)=\frac{i(i+1)}{(d+1)(d+2)}$, and 
\begin{align*}
\esp_{0,P_\xi,1}\Big( \frac{1}{d/2}\sum_{i=1}^{d/2} u_{(i)}^2-\frac{1}{d}\sum_{i=1}^d u_i^2\Big)&\le
\frac{1}{4}\esp_{0,P_\xi,1}(u^2_{(d/4)}-u^2_{(3d/4)}) = -\frac{d}{8(d+2)} \le -\frac{1}{24}.
\end{align*}
This and \eqref{43} prove the proposition. 


\section{Lemmas}

\subsection{Lemmas for the upper bounds}

\begin{lemma}\label{lemma_esp4_subgaussian}
	Let  $z_1,\ldots, z_d\simiid P$ with $P \in\mathcal{G}_{a,\tau}$ for some $a,\tau>0$ and let  
	$z_{(1)}\le\cdots\le z_{(d)}$ be the order statistics of $|z_1|,\ldots, |z_d|$. Then for $u>2^{1/a}\tau\vee 2$, we have
	\begin{equation}\label{eq:lemma_esp4_subgaussianA}
		\prob \Big( z_{(d-j+1)}\le u   \log^{1/a}\big(ed/j\big) , \forall\; j=1,\ldots, d \Big) \ge 1 - 4e^{-u^a/2},
	\end{equation}
{
	 and, for any $r>0$,
	\begin{equation}\label{eq:lemma_esp4_subgaussian}
		\esp \big(  z_{(d-j+1)}^r \big)\le C \log^{r/a}\big(ed/j\big), \qquad j=1,\dots, d,
	\end{equation}
	where $C>0$ is a constant depending only on $\tau$, $a$ and $r$.
}
\end{lemma}

\begin{proof}
	Using the definition of $\mathcal{G}_{a,\tau}$ we get that, for any $t\ge2$,
	\begin{equation*}
		\prob\big( z_{(d-j+1)}\ge t\big)\le 
\binom{d}{j}\prob^j(|z_1|\ge t)\le 2\Big(\frac{ed}{j}\Big)^j e^{-j(t/\tau)^a},\qquad j=1,\dots, d.
	\end{equation*}
{
Thus, for $v\ge 2^{1/a}\vee (2/\tau)$ we have
\begin{equation}\label{eqxx}
\prob (   z_{(d-j+1)}\ge v\tau \log^{1/a}({ed}/{j}))\le 2\Big(\frac{ed}{j}\Big)^{j(1-v^a)} \le 
2 e^{-jv^a/2},\qquad j=1,\dots, d,
\end{equation}
and 
	$$
	\prob \Big(   \exists \ j\in\{1,\ldots, d\}: 
z_{(d-j+1)}\ge v \tau\log^{1/a}(ed/j) \Big)\le  2\sum_{j=1}^d e^{-jv^a/2}\le 4e^{-v^a/2}
	$$
implying \eqref{eq:lemma_esp4_subgaussianA}. Finally, \eqref{eq:lemma_esp4_subgaussian}
follows by integrating \eqref{eqxx}. 
}
\end{proof}

\begin{lemma}\label{lemma_esp4_polynomial} Let  $z_1,\ldots, z_d\simiid P$ with $P \in\mathcal{P}_{a,\tau}$ for some $a,\tau>0$ and let  
	$z_{(1)}\le\cdots\le z_{(d)}$ be the order statistics of $|z_1|,\ldots, |z_d|$. Then
	  for  $u> (2 e)^{1/a} \tau\vee 2$,  we have 
	\begin{equation}\label{eq:lemma_esp4_polynomial}
		\prob \Big(   z_{(d-j+1)}\le u   \Big(\frac{d}{j} \Big)^{1/a} , \forall\; j=1,\ldots, d \Big) \ge  1-\frac{2 e \tau^a}{u^a}
	\end{equation}
	{
	and, for any $r\in (0,a)$,
	\begin{equation}\label{eq2:lemma_esp4_polynomial}
	\esp \big(  z_{(d-j+1)}^r \big)\le C \Big(\frac{d}{j}\Big)^{r/a}, \qquad j=1,\dots, d,
	\end{equation}
	where $C>0$ is a constant depending only on $\tau$, $a$ and $r$. 
	}
\end{lemma}
\begin{proof}
	Using the definition of $\mathcal{P}_{a,\tau}$ we get that, for any $t\ge2$,
	\begin{equation*}
		\prob\big( z_{(d-j+1)}\ge t\big)\le \Big(\frac{ed}{j}\Big)^j \Big(\frac{\tau}{t}\Big)^{ja}.
	\end{equation*}
	Set $t_j=u   \Big(\frac{d}{j} \Big)^{1/a}$ and $q=e(\tau/u)^a$. The assumption on $u$ yields that $q<1/2$, so that 
	$$
	\prob \Big(   \exists \ j\in\{1,\ldots, d\}: z_{(d-j+1)}\ge u   \Big(\frac{d}{j} \Big)^{1/a} \Big)\le \sum_{j=1}^d\Big(\frac{ed}{j}\Big)^j \Big(\frac{\tau}{t_j}\Big)^{ja} = \sum_{j=1}^d q^j\le 2q.
	$$
	This proves \eqref{eq:lemma_esp4_polynomial}. The proof of \eqref{eq2:lemma_esp4_polynomial} is analoguous to that of~\eqref{eq:lemma_esp4_subgaussian}.
\end{proof}
{
\begin{lemma}\label{lemma:sum} For all $a>0$ and all integers $1\le s\le d$,
$$\sum_{i=1}^{s} \log^{2/a}\big(ed/i\big) \le Cs \log^{2/a}\Big(\frac{ed}{s}\Big)$$
where $C>0$ depends only on $a$. 
\end{lemma}
The proof is simple and we omit it.
}

\subsection{Lemmas for the lower bounds}

For two probability measures ${\rm P}_1$ and ${\rm P}_2$ on a measurable space $(\Omega, \mathcal{U})$, we denote by $V({\rm P}_1,{\rm P}_2)$ the total variation distance between ${\rm P}_1$ and ${\rm P}_2$:
$$V({\rm P}_1,{\rm P}_2)=\sup_{B\in \mathcal{U}}\left|{\rm P}_1(B)-{\rm P}_2(B)\right|.$$

\begin{lemma}[Deviations of the binomial distribution]\label{binomial}
	Let $\mathcal{B}(d,p)$ denote the binomial random variable with parameters $d$ and~$p\in (0,1)$. 
	Then, for any $\lambda>0$,
	\begin{align}\label{binomial1}
		&\prob\big(\mathcal{B}(d,p)\ge\lambda \sqrt{d}+dp\big) \le \exp\bigg(-\frac{\lambda^{2}}{2p(1-p)\big(1+\frac{\lambda}{3p\sqrt{d}}\big)}\bigg),\\
		\label{binomial2}
		&\prob\big(\mathcal{B}(d,p)\le -\lambda \sqrt{d}+dp\big) \le \exp\bigg(-\frac{\lambda^{2}}{2p(1-p)}\bigg).
	\end{align}
\end{lemma}
	Inequality \eqref{binomial1} is a combination of formulas (3)  and (10) on pages 440--441 in~\cite{ShorackWellner1986}.
Inequality \eqref{binomial2} is formula (6) on page 440 in~\cite{ShorackWellner1986}.

\begin{lemma}\label{lemma_TV}
	Let ${\mathbb P}_\mu$ and ${\mathbb P}_{\bar{\mu}}$ be the probability measures defined  in~(\ref{definition_apriori}). The total variation distance between these two measures satisfies
	\begin{equation}\label{eq1:lemma_TV}
		V({\mathbb P}_\mu,{\mathbb P}_{\bar{\mu}}) \le \prob\Big(\mathcal{B}\Big(d,\frac{s}{2d}\Big)>s\Big) \le e^{-\frac{3s}{16}},
	\end{equation}
and	
\begin{equation}\label{eq2:lemma_TV}
		V({\mathbb P}_\mu,{\mathbb P}_{\bar{\mu}}) \le 1- \prob\Big(\mathcal{B}\Big(d,\frac{s}{2d}\Big)= 0\Big)- \prob\Big(\mathcal{B}\Big(d,\frac{s}{2d}\Big)= 1\Big).
	\end{equation}
\end{lemma}

\begin{proof}
	We have 
	$$
		V({\mathbb P}_\mu,{\mathbb P}_{\bar{\mu}}) = \sup_{B}\left|\int\prob_{\bt,U,1} (B)d\mu(\bt)-\int \prob_{\bt,U,1}(B)d\bar{\mu}(\bt)\right| 
		\le 
		\sup_{|f|\le 1}\left|\int f d\mu-\int f d\bar{\mu}\right| 
		=
		V(\mu,\bar{\mu}).
	$$
	Furthermore, $V(\mu,\bar{\mu}) \le \mu(\T_s^c)$ since for any Borel subset $B$ of $\RR^d$ we have $\big|\mu(B)-\bar{\mu}(B)\big|\le  \mu(B\cap\T_s^c)$. Indeed,
	$$ 
	\mu(B)-\bar{\mu}(B)\le \mu(B)-\mu(B\cap \T)= \mu(B\cap \T^c)
	$$
	and 
	$$
	\bar{\mu}(B)-\mu(B) = \frac{\mu(B\cap \T)}{\mu(\T)}-\mu(B\cap \T)- \mu(B\cap \T^c) \ge - \mu(B\cap \T^c).
	$$
	Thus,
	\begin{equation}\label{eq3:lemma_TV}
		V({\mathbb P}_\mu,{\mathbb P}_{\bar{\mu}}) \le \mu(\T_s^c) = \prob\Big(\mathcal{B}\Big(d,\frac{s}{2d}\Big)>s\Big).
	\end{equation}
	Combining this inequality with \eqref{binomial1} we obtain \eqref{eq1:lemma_TV}. To prove \eqref{eq2:lemma_TV}, we use again  \eqref{eq3:lemma_TV} and notice that $\prob\Big(\mathcal{B}\Big(d,\frac{s}{2d}\Big)>s\Big)\le \prob\Big(\mathcal{B}\Big(d,\frac{s}{2d}\Big)\ge 2\Big)$ for any integer $s\ge 1$.
	\end{proof}

\begin{lemma}\label{lemma_concentration_barmu}
	Let $\bar{\mu}$ be defined in \eqref{definition_barmu} with some $\a>0$.
	Then
	\begin{equation}\label{eq1:lemma_concentration_barmu}
		\bar{\mu} \Big(\|\bt\|_2< \frac{\a}{2}\sqrt{s} \Big) \le 2e^{-\frac{s}{16}},
	\end{equation}
	and, for all $s\le 32$,
	\begin{equation}\label{eq2:lemma_concentration_barmu}
		\bar{\mu} \Big(\|\bt\|_2< \frac{\a\sqrt{s}}{4\sqrt{2}} \Big) = \prob\Big(\mathcal{B}\big(d,\frac{s}{2d}\big)= 0\Big).
	\end{equation}
\end{lemma}

\begin{proof}
First, note that
\begin{equation}\label{eq:lem:barmu}
		\mu \Big(\|\bt\|_2< \frac{\a}{2}\sqrt{s} \Big) = \prob\Big(\mathcal{B}\big(d,\frac{s}{2d}\big)< \frac{s}{4}\Big) \le e^{-\frac{s}{16}}
	\end{equation}
	where the last inequality follows from \eqref{binomial2}. Next, inspection of the proof of Lemma~\ref{lemma_TV} yields that $\bar{\mu}(B)\le {\mu}(B) + e^{-\frac{3s}{16}}$ for any Borel set~$B$. Taking here $B= \{\|\bt\|_2\le \a\sqrt{s}/2\}$ and using~\eqref{eq:lem:barmu} proves \eqref{eq1:lemma_concentration_barmu}. To prove \eqref{eq2:lemma_concentration_barmu}, it suffices to note that $\mu \Big(\|\bt\|_2< \frac{\a\sqrt{s}}{4\sqrt{2}} \Big) = \prob\Big(\mathcal{B}\big(d,\frac{s}{2d}\big)< \frac{s}{32}\Big)$. 
\end{proof}

\begin{lemma}\label{lemma_density}
There exists a probability density $f_0:\RR\to [0, \infty)$ with the following properties: $f_0$ is continuously differentiable, symmetric about 0, supported on $[-3/2,3/2]$, with variance~1 and finite Fisher information $I_{f_0}= \int (f_0'(x))^2(f_0(x))^{-1}dx$. 
\end{lemma}

\begin{proof}
Let $K:\RR\to [0, \infty)$ be any probability density, which is continuously differentiable, symmetric about 0, supported on $[-1,1]$, and has finite Fisher information $I_K$, for example, the density $K(x)= \cos^2(\pi x/2) \fcar_{|x|\le 1}$. Define
$f_0(x)= [K_h(x+(1-\varepsilon)) + K_h(x-(1-\varepsilon))]/2$ where $h>0$ and $\varepsilon \in (0,1)$ are constants to be chosen, and $K_h(u)=K(u/h)/h$. Clearly, we have $I_{f_0}<\infty$ since $I_{K}<\infty$.  It is straightforward to check that the variance of $f_0$ satisfies $\int x^2 f_0(x) dx = (1-\varepsilon)^2 + h^2 \s_K^2$ where $\s_K^2  = \int u^2 K(u)du$. Choosing $h=\sqrt{2\varepsilon-\varepsilon^2}/\s_K$  and $\varepsilon  \le \s_K^2/8$ guarantees that $\int x^2 f_0(x) dx = 1$ and the support of $f_0$ is contained in   $[-3/2,3/2]$.
\end{proof}
{
\begin{lemma}\label{lemma:lowerbound:norm:variance}
       Let $\tau>0$, $a>4$ and let $s,d$ be integers satisfying $1\leq s \leq d$. Let 
$\mathcal{P}$ be any subset of $\mathcal{P}_{a,\tau}$. Assume that for some function $\phi(s,d)$ 
of $s$ and $d$ and for some positive constants $c_{1},c_{2},c'_{1},c'_{2}$ we have
       \begin{equation}\label{eq1:lemma:lowerbound:norm:variance}
       \underset{\hat{T}}{\inf} \underset{P_\xi \in \mathcal{P}}{\sup}\,
       \underset{\s>0}{\sup}\,\underset{\|\bt\|_{0}\leq s}{\sup}\mathbf{P}_{\bt, P_{\xi},\sigma}
\left( \left| \frac{\hat{T}}{\s^{2}}-1\right|\geq \frac{c_{1}}{\sqrt{d}}\right) \geq c_1^{'},
       \end{equation}
       and
       \begin{equation}\label{eq2:lemma:lowerbound:norm:variance}
       \underset{\hat{T}}{\inf} \underset{P_\xi \in \mathcal{P}}{\sup}\,
       \underset{\s>0}{\sup}\, \underset{\|\bt\|_{0}\leq s}{\sup}\mathbf{P}_{\bt, P_{\xi},\sigma}\left( \left| \frac{\hat{T}-\|\bt\|_{2}}{\s}\right|\geq c_{2}\phi(s,d)\right) \geq c_2^{'}.
       \end{equation}
       Then 
       $$
       \underset{\hat{T}}{\inf} \underset{P_\xi \in \mathcal{P}}{\sup}\,
       \underset{\s>0}{\sup}\, \underset{\|\bt\|_{0}\leq s}{\sup}\mathbf{P}_{\bt, P_{\xi},\sigma}\left( \left| \frac{\hat{T}}{\s^{2}}-1\right|\geq c_{3}\max\left(\frac{1}{\sqrt{d}},\frac{\phi^{2}(s,d)}{d}\right)\right) 
\geq c_3^{'}
       $$
       for some constants $c_{3},c_{3}'>0$.
       \end{lemma}
       
\begin{proof} 
       Let $\hat{\s}^{2}$ be an arbitrary estimator of $\s^{2}$. Based on $\hat{\s}^{2}$, we can
construct an estimator $\hat{T}= \hat{N}^*$ of $\|\bt\|_{2}$ defined by formula \eqref{eq:C}, 
case $s>\sqrt{d}$.
  It follows from \eqref{g2}, \eqref{g3} and \eqref{eq2:lemma:lowerbound:norm:variance} that
       \begin{align*}
       c_{2}'&\leq \mathbf{P}\left(2|(\bt,\bxi)| \geq c_{2}\|\bt\|_{2}\phi(s,d)/3\right) + \mathbf{P}\left(\sqrt{|\|\bxi\|_{2}^{2}-d|} \geq c_{2}\phi(s,d)/3\right)
\\
& \qquad + \mathbf{P}\left( \sqrt{d\left|\frac{\hat{\s}^{2}}{\s^{2}}-1\right|}
\geq c_{2}\phi(s,d)/3\right),
       \end{align*}
    where we write for brevity $\mathbf{P} =\mathbf{P}_{\bt, P_{\xi},\sigma}$.   Hence
       $$
       \mathbf{P}\left( \left|\frac{\hat{\s}^{2}}{\s^{2}}-1\right|
\geq c_{2}^{2}\phi^{2}(s,d)/(9d)\right) \geq c_{2}'-c^{*}\max\left(\frac{d }{\phi^{4}(s,d)},
\frac{1}{\phi^{2}(s,d)}\right)
       $$
       for some constant $c^{*}>0$ depending only on $a$ and $\tau$. If $\phi^{2}(s,d) > 
\max\left( \sqrt{\frac{2c^{*}d}{c_{2}'}},\frac{2c^{*}}{c_{2}'} \right)$, then
       $$
       \mathbf{P}\left( \left|\frac{\hat{\s}^{2}}{\s^{2}}-1\right|\geq C 
\max\left( \frac{1}{\sqrt{d}},\frac{\phi^{2}(s,d)}{d}\right)\right) \geq c_{2}'/2.
       $$
        If 
$\phi^{2}(s,d) \le \max\left( \sqrt{\frac{2c^{*}d}{c_{2}'}},\frac{2c^{*}}{c_{2}'} \right)$, then
$\max\left( \frac{1}{\sqrt{d}},\frac{\phi^{2}(s,d)}{d}\right)$ is of order $\frac{1}{\sqrt{d}}$ and
the result follows from \eqref{eq1:lemma:lowerbound:norm:variance}.
       \end{proof}

}
\section{Acknowledgements}
The work of O.~Collier has been conducted as part of the project Labex MME-DII (ANR11-LBX-0023-01).
The work of M.Ndaoud and A.B. Tsybakov was supported by GENES and by the French National Research
    Agency (ANR) under the grants IPANEMA (ANR-13-BSH1-0004-02) and Labex Ecodec (ANR-11-LABEX-0047).

\end{document}